\pdfoutput=1 
\documentclass[a4paper]{amsart}

\usepackage[T1]{fontenc}
\usepackage{lmodern}
\usepackage{microtype}

\usepackage{hyperref}

\usepackage{amssymb}

\usepackage{graphicx}

\usepackage[abbrev]{amsrefs}
\usepackage{enumerate}

\newcommand{\abs}[1]{\mathopen\lvert#1\mathclose\rvert}
\newcommand{\bigabs}[1]{\bigl\lvert#1\bigr\rvert}
\newcommand{\Bigabs}[1]{\Bigl\lvert#1\Bigr\rvert}

\newcommand{\seminorm}[1]{\mathopen\mathbf{\lvert}#1\mathclose\mathbf{\rvert}}
\newcommand{\norm}[1]{\mathopen\lVert#1\mathclose\rVert}

\newcommand{\Nset}{{\mathbb N}}
\newcommand{\Rset}{{\mathbb R}}
\newcommand{\Cset}{{\mathbb C}}
\newcommand{\Sset}{{\mathbb S}}
\newcommand{\Bset}{{\mathbb B}}
\newcommand{\Hset}{{\mathbb H}}
\newcommand{\mobius}{\mathcal{M}}

\DeclareMathOperator{\dist}{dist}

\DeclareMathOperator{\acosh}{acosh}
\newcommand{\scalprod}[2]{ #1 \cdot #2 }

\newcommand{\st}{\;:\;}

\newcommand{\dif}{\,\mathrm{d}}

\theoremstyle{plain}
\newtheorem{proposition}{Proposition}[section]
\newtheorem{lemma}[proposition]{Lemma}
\newtheorem{theorem}{Theorem}

\theoremstyle{definition}

\theoremstyle{remark}
\newtheorem{remark}{Remark}[section]

\newcommand{\step}[1]{\medbreak \noindent \textbf{#1}}

\title[Controlled singular extension of critical Sobolev maps]
{Controlled singular extension of critical trace Sobolev maps
from spheres to compact manifolds}



\date{\today}

\author{Mircea Petrache}
\address{
Mircea Petrache\\
 UPMC Universit\'e Paris 6, UMR 7598\\
 Laboratoire Jacques-Louis Lions\\
Place Jussieu 4\\
75005 Paris\\
France}
\email{mircea.petrache@upmc.fr}

\author{Jean Van Schaftingen}
\address{
Jean Van Schaftingen\\
 Universit\'e catholique de Louvain\\
 Institut de Recherche en Math{\'e}matique et Physique\\
 Chemin du cyclotron 2, bte L7.01.01\\
1348 Louvain-la-Neuve\\
Belgium}
\email{Jean.VanSchaftingen@UCLouvain.be}

\begin{document}


\begin{abstract}
Given $n \in \Nset_*$, a compact Riemannian manifold $M$ 
and a Sobolev map $u \in W^{n/(n + 1), n + 1} (\mathbb{S}^n; M)$, 
we construct a map  $U$ in the Sobolev--Marcin\-kiewicz (or Lorentz--Sobolev) space $W^{1, (n + 1, \infty)} (\mathbb{B}^{n + 1}; M)$ 
such that  $u = U$ in the sense of traces on $\mathbb{S}^{n} = \partial \mathbb{B}^{n + 1}$ 
and whose derivative is controlled: 
for every $\lambda > 0$,
$$
\lambda^{n + 1} \big\vert\big\{ x \in \mathbb{B}^{n + 1} \st \abs{D U (x)} > \lambda\big\}\big\vert 
  \le \gamma \Big(\int_{\mathbb{S}^n}\int_{\mathbb{S}^n} \frac{\abs{u (y) - u (z)}^{n + 1}}{\abs{y - z}^{2 n}} \,\mathrm{d} y \,\mathrm{d} z \Bigr)\ ,
$$
where the function $\gamma : [0, \infty) \to [0, \infty)$ 
only depends on the dimension $n$ and on the manifold $M$.
The construction of the map $U$ relies on a smoothing process by hyperharmonic extension 
and radial extensions on a suitable covering by balls.
\end{abstract}

\maketitle

\section{Introduction}

\subsection{Traces of Sobolev maps}
For any \(p \in (1, \infty)\), the classical Sobolev trace theory states that a function \(u : \Sset^n \to \Rset^m\) is the trace of a function \(U \in W^{1, p} (\Bset^{n + 1}; \Rset^m)\) if and only if 
\(
 u \in W^{1 - 1/p, p} (\Sset^n; \Rset^m),
\) \cite{DiBenedetto2002}*{\S 18},
where, for \(s \in (0, 1)\), the fractional Sobolev space \(W^{s, p} (\Sset^n; \Rset^m)\) is  the set of functions \(u \in L^p (\Sset^{n}; \Rset^m)\) whose Gagliardo seminorm is finite:
\[
  \seminorm{u}_{W^{s, p}}^p = \int_{\Sset^n}\int_{\Sset^n} \frac{\abs{u (y) - u (z)}^p}{\abs{y - z}^{n + s p}} \dif y \dif z < \infty\ .
\]

The corresponding problem when the Euclidean space \(\Rset^m\) 
is replaced by a compact Riemannian manifold \(M\), which is without loss of generality embedded into \(\Rset^\nu\) by the classical Nash embedding theorem \cite{Nash1956}, is more delicate.
When \(p > n + 1\), the maps \(u \in W^{1-1/p, p} (\Sset^n; M)\) are continuous.
The maps \(u : \Sset^n \to M\) that are the trace of a function \(U \in W^{1, p} (\Bset^{n + 1}; M)\) are then the maps \(u \in W^{1-1/p, p} (\Sset^n; M)\) which are homotopic to a constant. 
In particular, if \(\pi_n (M) = \{0\}\), that is, the \(n\)--th homotopy group of \(M\) is trivial, then every map in \(W^{1-1/p, p} (\Sset^n; M)\) is the trace of a map in \(W^{1, p} (\Bset^{n + 1}; M)\) \cite{BethuelDemengel1995}*{theorem 1}.
When \(p = n + 1\), a similar conclusion can be drawn by working with maps of vanishing mean oscillation instead of continuous maps \cite{BethuelDemengel1995} (following an idea of \cite{SchoenUhlenbeck}).
When \(p < n + 1\) and \(M\) is simply connected, then each map in \(W^{1-1/p, p} (\Sset^n; M)\) is the trace of a map in \(W^{1, p} (\Bset^{n + 1}; M)\) if and only if \(M\) is \(\lfloor p - 1\rfloor\)--simply connected: 
for every \(j \in \Nset_*\) such that \(j \le p - 1\), \(\pi_\ell (M) \simeq \{0\}\) \citelist{\cite{Bethuel2014}*{theorem 1.1}\cite{HardtLin1987}*{theorem 6.2}} (see also \cite{ Isobe2003}).
Many things are also known when \(M\) is not simply connected \cite{Bethuel2014}.

\subsection{Controlled extension}
In the case \(M = \Rset^m\), the extension map \(u \mapsto U\) can be taken to be linear --- by taking for example the harmonic extension of \(u\) to the ball \(\Bset^n\) --- and then the Gagliardo seminorm of \(u\) \emph{controls linearly} the Sobolev seminorm of its extension \(U\):
\begin{equation}
\label{eqLinearControl}
 \int_{\Rset^m} \abs{D U}^p
 \le C \int_{\Sset^n}\int_{\Sset^n} \frac{\abs{u (y) - u (z)}^p}{\abs{y - z}^{n + p - 1}} \dif y \dif z\ .
\end{equation}

A natural question is whether each map in \(W^{1 - 1/p, p} (\Sset^n; M)\) has an extension  that is \emph{controlled} in \(W^{1, p}(\Bset^{n + 1}; M)\). 
When \(p < n + 1\) and \(M\) is \(\lfloor p - 1\rfloor\)--simply connected, the construction of the extension by Hardt and Lin \cite{HardtLin1987}*{theorem 6.2} yields as a byproduct an estimate of the form \eqref{eqLinearControl}.
When \(p > n + 1\), by a compactness argument(see \cite{PetracheRiviere2015}*{proposition 2.8} and \S \ref{sectionConstantRegularity} below), 
there is a function \(\gamma : [0, \infty) \to [0, \infty)\) such that if \(u \in W^{1 - 1/p,p} (\Sset^n; M)\) is homotopic to a constant, 
then there exists \(U \in W^{1, p} (\Bset^{n + 1}; M)\) whose trace on \(\Sset^n\) is \(u\) and such that 
\[
  \int_{\Bset^{n + 1}} \abs{D U}^{p}
  \le \gamma \Bigl(\int_{\Sset^n}\int_{\Sset^n} \frac{\abs{u (y) - u (z)}^p}{\abs{y - z}^{n + p - 1}} \dif y \dif z \Bigr)\ .
\]
The norm of the extension is controlled, but \emph{not linearly}. 

This result does not extend to the critical case \(p = n + 1\) when \(M = \Sset^n\), 
due to the existence of a sequence of smooth maps homotopic to a constant that is bounded in \(W^{n/(n + 1), n + 1}(\Sset^{n}; \Sset^n)\) 
and that converges almost everywhere to a smooth map which is not homotopic to constant \cite{PetracheRiviere2015}*{proposition 2.8}.

The first author and Rivi\`ere \cite{PetracheRiviere2015} have proposed to control the derivative in the weak \(L^{n + 1}\) or Marcinkiewicz space \(L^{n + 1, \infty} (\Bset^{n + 1})\), whose quasinorm is defined for \(f : \Bset^{n + 1} \to \Rset\) by 
\[
 \norm{f}_{L^{n + 1, \infty}(\Bset^{n + 1})}
 = \sup_{\lambda > 0}\lambda^{n + 1} \bigabs{\bigl\{ x \in \Bset^{n + 1} \st \abs{f (x)} > \lambda\bigr\}}\ ,
\]
and thus to construct the extension in the Sobolev--Marcinkiewicz space 
\[
 W^{1, (n, \infty)} (\Bset^{n + 1}; M)=
 \bigl\{ U \in W^{1, 1}_{\mathrm{loc}} (\Bset^{n + 1}); M) \st \norm{D U}_{L^{n + 1, \infty} (\Bset^{n + 1})}
 <\infty \bigr\}\ ,
\]
which is the smallest space in the Lorentz--Sobolev space scale
in which radial extensions of maps inside the ball are contained.
When \(n \in \{1, 2, 3\}\), 
they have proved that there exists a function \(\gamma : [0, \infty) \to [0, \infty)\) 
such that for each \(u \in W^{1, n} (\Sset^n; \Sset^n)\), 
there exists \(U \in W^{1, (n + 1, \infty)} (\Bset^{n + 1}; \Sset^n)\) 
whose trace is \(u\) and such that for every \(\lambda > 0\)
\begin{equation}
\label{ineqPetracheRiviere}
\lambda^{n + 1} \bigabs{\bigl\{ x \in \Bset^{n + 1} \st \abs{D U (x)} > \lambda\bigr\}} \le \gamma \Bigl(\int_{\Sset^n} \abs{D u}^n \Bigr)\ .
\end{equation}
In dimension \(n = 2\), their proof relies on a Hopf lift of the map \(u\); 
when \(n = 3\) they use the group structure in the target manifold \(\Sset^3 = SU (2)\).
Their proof depends thus strongly on the dimensions and relies on the structure of the target manifold.

In the present work, we obtain the following controlled extension.

\begin{theorem}
\label{theoremMain}
Let \(n \in \Nset_*\) and let \(M \subset \Rset^{\nu}\) be a compact embedded manifold.
There exists a function \(\gamma : [0, \infty) \to [0, \infty)\) 
such that for every \(u \in W^{n/(n + 1), n + 1} (\Sset^n; M)\), 
there exists \(U \in W^{1, (n + 1, \infty)} (\Bset^{n + 1}; M)\) 
such that  \(u = U\) in the sense of traces and for every \(\lambda > 0\),
\[
\lambda^{n + 1} \bigabs{\bigl\{ x \in \Bset^{n + 1} \st \abs{D U (x)} > \lambda\bigr\}} 
  \le \gamma \Bigl(\int_{\Sset^n}\int_{\Sset^n} \frac{\abs{u (y) - u (z)}^{n + 1}}{\abs{y - z}^{2 n}} \dif y \dif z \Bigr)\ .
\]
\end{theorem}

Compared to the result of \cite{PetracheRiviere2015}, 
we cover maps from a sphere of \emph{any dimension} 
to \emph{any compact manifold} and the extension is estimated
in the \emph{natural trace space}.
By classical Sobolev embedding theorems, theorem~\ref{theoremMain} implies the estimate \eqref{ineqPetracheRiviere} and answers positively the problem of extending it to higher dimensions \cite{PetracheRiviere2015}*{open problem 1.5}.
We improve thus the result of \cite{PetracheRiviere2015} both at the geometric level and at the analytic level.
The price to pay for this generality is that the function \(\gamma\) is an exponential of an exponential.

Theorem~\ref{theoremMain} can be restated without relying on traces by stating 
that any smooth map can be approximated by smooth map except at finitely many points 
whose classical derivative satisfies the weak-type estimate. 
In one direction this comes from the approximation of a map 
\(W^{n/(n + 1), n + 1}(\Sset^{n}; M)\) by smooth maps 
from \(\Sset^n\) to \(M\) --- essentially because the maps have vanishing mean oscillation, 
otherwise the density of smooth maps depends on the topology of \(M\) \cite{BrezisMironescu2015}.
Conversely, a smooth map can be obtained 
from as a byproduct of our construction (remark~\ref{remarkSmooth}).

If, up to a bilipschitz deformation of the ambient space \(\mathbb R^\nu\),
\(M\) is a Lipschitz retract of its convex hull then the situation becomes much simpler.
An extension with a \emph{linear}  seminorm estimate can be obtained by projection (see remark \ref{remarkRetraction}). 
The main difficulty is thus created by the nontrivial topology of \(M\).

Since both quantities
\begin{align*}
 &\int_{\Sset^n}\int_{\Sset^n} \frac{\abs{u (y) - u (z)}^{n + 1}}{\abs{y - z}^{2 n}} \dif y \dif z &
 &\text{ and }
 &\int_{\Bset^{n + 1}} \abs{D U}^{n + 1}
\end{align*}
are invariant under composition of \(u\) and \(U\) with M\"obius transformations of the ball \(\Bset^{n + 1}\), that is, 
isometries of the hyperbolic space \(\Hset^{n + 1}\) in the Poincar\'e ball model (see \S \ref{sectionHyperbolic}), 
we perform our construction in the Poincar\'e ball model of the hyperbolic space \(\Hset^{n + 1}\)
and we begin by taking \(\mathfrak{H}u : \Hset^{n + 1} \to \Rset^\nu\) to be the
hyperharmonic extension of \(u\), that is the harmonic extension with respect to the hyperbolic metric  \cite{Ahlfors1981}*{chapter V} (\S \ref{sectionConformalExtension}).
This extension is compatible with the action of the M\"obius group.
The M\"obius group was already used for \(n = 3\) and \(M = \Sset^3\) \cite{PetracheRiviere2015}*{\S 4C}.

In general, the map \(\mathfrak{H}u\) does not take its values in the manifold \(M\), but rather just in its convex hull in \(\mathbb R^\nu\).
However it is well known that \(\mathfrak{H} u\) takes values close to \(M\) in a neighborhood of the sphere \(\Sset^{n}\).
In order to obtain a controlled extension far from the boundary, we prove a new estimate on the proportion of hyperbolic spheres whose image is not close to the manifold \(M\) (proposition~\ref{propositionGoodRadiusFractional}): for every \(R \in (0 ,1)\), one has
\begin{multline*}\int_0^R  \frac{\norm{\dist (\mathfrak{H} u, u (\Sset^n))}_{L^\infty (\partial\Bset_r)}}{1 - r^2}\dif r\\
 \le C \Bigl(1 + \ln \frac{1}{1 - R^2} \Bigr)^\frac{n}{n + 1} \left(\int_{\Sset^n} \int_{\Sset^n} \frac{\abs{u (y) - u (z)}^{n + 1}}{\abs{y - z}^{2 n}}\dif y \dif z \right)^\frac{1}{n + 1} .
\end{multline*}

We construct then the map \(U\) by performing radial extensions from spheres whose image is close to the manifold \(M\) (see \S \ref{sectionHomogeneous}) along a suitable covering by balls, 
projecting the resulting map back to \(M\) and transferring the estimates from the hyperbolic space \(\Hset^{n + 1}\) back to the Euclidean ball \(\Bset^{n + 1}\) (\S \ref{sectionMainProof}).
This strategy of smoothing on a part of the domain and performing radial extensions on another part of the domain goes back to the approximation of Sobolev maps by smooth maps by Bethuel \cite{Bethuel1991}.

The proof of theorem is robust enough to go through when \(M\) is a general subset 
of \(\mathbb R^\nu\) which has a Lipschitz retraction from a uniform neighborhood to itself, 
(see remark \ref{remarkGromovSchoen}). 

As explained above, theorem~\ref{theoremMain} gives the estimate \eqref{ineqPetracheRiviere} by classical Sobolev embedding theorems. 
In \S \ref{sectionConstantRegularity}, we also show how our construction 
works when one is only interested in the estimate \eqref{ineqPetracheRiviere}.

\section{Hyperbolic geometry and M\"obius invariance}
\label{sectionHyperbolic}

Our construction of the extension will be simplified considerably by the observation that the 
integral
\[
 \int_{\Bset^{n + 1}} \abs{D U}^{n + 1}\ ,
\]
is \emph{conformally invariant}, that is, the quantity remains invariant if the canonical Euclidean metric on the ball \(\Bset^{n + 1}\) is replaced by a metric which pointwise is a scalar multiple of it.

In order to work with a richer group of isometries, we shall use the metric \(g_{\Hset^{n + 1}}\) of the classical \emph{Poincar\'e ball model of the hyperbolic space \(\Hset^{n + 1}\)}. 
This metric \(g_{\Hset^{n + 1}}\) is defined for each \(x \in \Bset^{n + 1}\) and \(v, w \in T_x \Bset^{n + 1} \simeq \Rset^{n + 1}\) by
\[
  g_{\Hset^{n + 1}} (x)[v, w] = \frac{4 v \cdot w}{(1 - \abs{x}^2)^2}\ .
\]
The associated hyperbolic distance \(d_{\Hset^{n + 1}}\) between two points \(x, y \in \Bset^{n + 1}\) can be computed explicitly by 
\[
  d_{\Hset^{n + 1}} (x, y) = \acosh \Bigl(1 + 2\frac{\abs{x - y}^2}{(1 - \abs{x}^2)(1 - \abs{y}^2)}\Bigr) ,
\]
while the associated measure \(\mu_{\Hset^{n + 1}}\) is given for any Lebesgue measurable set \(A \subseteq \Bset^{n + 1}\) by  
\begin{equation}
\label{eqHyperbolicMeasure}
 \mu_{\Hset^{n + 1}} (A) = \int_{A} \frac{2^{n + 1}}{(1 - \abs{x}^2)^{n + 1}} \dif x\ ,
\end{equation}
and the dual norm \(\abs{\cdot}_{\Hset^{n + 1}}\) is defined for \(x \in \Hset^{n + 1} \simeq \Bset^{n + 1}\) and \(\theta \in T^*_x \Hset^{n + 1} \simeq (\Rset^{n + 1})^*\) by 
\begin{equation}
\label{eqHyperbolicDualNorm}
  \abs{\theta}_{\Hset^{n + 1}} = \frac{1 - \abs{x}^2}{2} \abs{\theta}\ .
\end{equation}
In particular, it follows that for each \(U \in W^{1, 1}_{\mathrm{loc}} (\Bset^{n + 1})\),
\[
  \int_{\Bset^{n + 1}} \abs{D U}^{n + 1} = \int_{\Hset^{n + 1}} \abs{D U}^{n + 1}_{\Hset^{n + 1}} \dif \mu_{\Hset^{n + 1}}\ ,
\]
which was already a consequence of the conformal invariance mentioned above.

The isometries of the hyperbolic space correspond to the group of M\"obius transformations of \(\Rset^{n + 1}\) that preserve the unit ball \(\Bset^{n + 1}\), or, equivalently, the unit sphere \(\Sset^{n}\). 
\label{sectionMoebiusInvariance}
This M\"obius group of conformal transformations that keep the ball invariant can be described as \cite{Ahlfors1981}*{\S 2.6}. 
\[
 \mobius (\Bset^{n + 1}) 
 =\left\{ \begin{aligned}T : \Bset^{n + 1} \to \Bset^{n + 1} \st \forall x \in \Bset^{n + 1}\ T (x)= R \Bigl(\frac{(1 - \abs{a}^2)(x - a) - \abs{x - a}^2 a}{1 + \abs{x}^2\abs{a}^2 - 2\scalprod{x}{a}} \Bigr)\\
 \text{with \(R \in O (n + 1)\) and \(a \in \Bset^{n + 1}\)} 
 \end{aligned}\right\}\ .
\]
This group is isomorphic to the one obtained by restricting its elements \(T\) defined as above to the unit sphere \(\Sset^{n}\). 
On the sphere this becomes
\[
 \mobius (\Sset^{n}) 
 =\left\{\begin{aligned} T : \Sset^{n} \to \Sset^{n} \st  \forall x \in \Sset^{n}\  T(x) = R \Bigl(\frac{(1 - \abs{a}^2)(x - a) - \abs{x - a}^2 a}{\abs{x-a}^2} \Bigr) \\ 
 \text{with \(R \in O (n + 1)\) and \(a \in \Bset^{n + 1}\)}
\end{aligned}\right\}\ .
\]
For \(a \in \Bset^{n + 1}\), we define the transformation \(T_a \in \mobius (\Bset^{n + 1})\) for \(x \in \Bset^{n + 1}\) by 
\begin{equation}
\label{eqDefinitionTa}
 T_a (x) = \frac{(1 - \abs{a}^2)(x - a) - \abs{x - a}^2 a}{1 + \abs{x}^2\abs{a}^2 - 2\scalprod{x}{a}}\ .
\end{equation}
Geometrically, \(T_a\) is the the hyperbolic translation that maps the point \(a\) to the point \(0\) along the line passing through both points.

For every \(T_a \in \mobius (\Sset^{n})\), \(T_a\) is differentiable and for each \(x \in \Sset^{n}\) and \(v \in \Rset^{n + 1}\), we have \cite{Ahlfors1981}*{I, (30) and (34)} 
\begin{equation}
\label{eqMobiusDerivative}
 \abs{D T_a (x) [v]} = \abs{D T_a (x)}\abs{v}= \frac{1 - \abs{a}^2}{|x-a|^2} \abs{v} = \frac{1 - \abs{T_a (x)}^2}{1 - \abs{x}^2} \abs{v}\ .
\end{equation}
In general for \(T \in \mobius (\Bset^{n + 1})\), we have 
\begin{equation}
\label{eqMobiusGeneralDerivative}
\abs{D T (x)} = \frac{1 - \abs{T (x)}^2}{1 - \abs{x}^2}\ .
\end{equation}
It can be also observed that \cite{Ahlfors1981}*{I, (32)}
\begin{equation}
\begin{split}
  \abs{T (x) - T (y)}^2 &= \frac{(1 - |a|^2)^2}{\bigl(1 + \abs{x}^2\abs{a}^2 - 2\,\scalprod{x}{a}\bigr) \bigl(1 + \abs{y}^2\abs{a}^2 - 2\,\scalprod{y}{a}\bigr)}\abs{x - y}^2\\
  &=\abs{D T (x)}\abs{D T (y)} \abs{x - y}^2\ .
\end{split}
\end{equation}
In view of the definition of the M\"obius group \(\mobius (\Bset^{n + 1})\), for every \(T \in \Bset^{n + 1}\),
\begin{equation}
\label{eqMobiusGeneralDifference}
 \abs{T (x) - T (y)}^2 = \abs{D T (x)}\, \abs{D T (y)}\, \abs{x - y}^2\ .
\end{equation}
In particular, for every \(p \in (0, \infty)\) and any measurable function \(u : \Sset^n \to \Rset^\nu\)
\begin{equation}
\label{eqGagliardoConformal}
 \int_{\Sset^n} \int_{\Sset^n} \frac{\abs{u (x) - u (y)}^p}{\abs{x - y}^{2 n}} \dif x \dif y
 = \int_{\Sset^n} \int_{\Sset^n} \frac{\abs{(u \circ T) (x) - (u \circ T) (y)}^p}{\abs{x - y}^{2 n}} \dif x \dif y\ ;
\end{equation}
that is, the Gagliardo fractional seminorm on \(\Sset^n\) induced by the Euclidean 
distance on \(\Rset^{n + 1}\) is also invariant under the M\"obius group.

We close this section by warning the reader that the Sobolev--Marcinkiewicz quasinorm appearing in the conclusion of theorem~\ref{theoremMain} is \emph{not} conformally invariant. In fact, the quantities
\begin{equation*}
 \sup_{\lambda > 0} \lambda^{n + 1} \bigabs{\bigl\{ x \in \Bset^{n + 1} \st \abs{D U (x)} > \lambda\bigr\}} 
\end{equation*}
and
\begin{equation*}
 \sup_{\lambda > 0} \lambda^{n + 1} \mu_{\Hset^{n + 1}} \bigl(\{ x \in \Hset^{n + 1} \st \abs{D U (x)}_{\Hset^{n + 1}} > \lambda\}\bigr)
\end{equation*}
are not related to each other by any inequality.
We shall overcome this difficulty in the proof of theorem~\ref{theoremMain} by 
showing that these quantities control 
each other on hyperbolic balls of controlled radius (see \eqref{eqHyperbolicCloseIsometry}).

\section{Hyperharmonic extension}
\label{sectionConformalExtension}
Given \(u \in L^1 (\Sset^n; \Rset^\nu)\), we define its hyperharmonic extension \(\mathfrak{H} u : \Bset^{n + 1} \to \Rset^\nu\) for \(x \in \Bset^{n + 1}\) by 
\begin{equation}
\label{equationConformalExtension}
 \mathfrak{H} u (x) =  \frac{1}{\abs{\Sset^n}}\int_{\Sset^{n}} u \circ T_x^{-1}
 =  \frac{(1 - \abs{x}^2)^{n}}{\abs{\Sset^n}} \int_{\Sset^{n}} \frac{u (y)}{\abs{y - x}^{2 n}} \dif y\ ,
\end{equation}
where the M\"obius transformation \(T_x\) was defined in \eqref{eqDefinitionTa} and its Jacobian determinant was computed with \eqref{eqMobiusDerivative} \cite{Ahlfors1981}*{(27)}.

When \(n = 1\), \(\mathfrak{H} u\) is also the harmonic extension of \(u\) (the kernel coincides with the Poisson kernel); whereas when \(n = 3\), \(\mathfrak{H}u\) is biharmonic in \(\Bset^{n + 1}\) and its normal derivative vanishes on the boundary \citelist{\cite{Nicolesco1936}\cite{SchulzeWildenhain1977}*{VIII.9.2}}.

If \(u = 1\), we obtain from \eqref{equationConformalExtension} the integral identity
\begin{equation}
\label{eqKernelConstant}
 \frac{1}{\abs{\Sset^{n}}} \int_{\Sset^n} \frac{(1 - \abs{x}^2)^n}{\abs{x - y}^{2 n}}\dif y = 1\ .
\end{equation}

\begin{lemma}[M\"obius covariance of the hyperharmonic extension]
If \(T \in \mobius (\Bset^{n + 1})\), we have 
\[
 \mathfrak{H} (u \circ T) = (\mathfrak{H}u) \circ T\ .
\]
\end{lemma}
\begin{proof}
For \(x \in \Bset^{n + 1}\), we have in view of the definition \eqref{equationConformalExtension}
\begin{equation}
\label{eqConformalDecomposition}
 \mathfrak{H} (u \circ T) (x) 
 = \frac{1}{\abs{\Sset^n}}\int_{\Sset^{n}} (u \circ T) \circ T_x^{-1}
 = \frac{1}{\abs{\Sset^n}}\int_{\Sset^{n}} (u \circ T_{T (x)}^{-1}) \circ (T_{T (x)} \circ T \circ T_x^{-1})\ .
\end{equation}
We observe that 
\[
T_{T (x)} (T (T_x^{-1}(0)))
= T_{T (x)} (T (x)) = 0\ ,
\]
and so the transformation \(T_{T (x)}^{-1} \circ T \circ T_x^{-1} \in \mobius (\Bset^{n + 1})\) preserves the point \(0\). Hence \(T_{T (x)}^{-1} \circ T \circ T_x^{-1}\in O(n+1)\) and thus, by a change of variable on the sphere \(\Sset^{n}\) we conclude from \eqref{eqConformalDecomposition} that 
\[
 \frac{1}{\abs{\Sset^n}}\int_{\Sset^{n}} (u \circ T) \circ T_x^{-1}
= \frac{1}{\abs{\Sset^n}}\int_{\Sset^{n}} u  \circ T_{T (x)}^{-1} = (\mathfrak{H} u) (T (x))\ .\qedhere
\]
\end{proof}

\begin{proposition}[Analytic properties of the hyperharmonic extension]
\label{propositionConformalExtension}Let \(u \in W^{\frac{n}{n + 1}, n + 1} (\Sset^{n})\).
\begin{enumerate}[(i)]
\item \label{estimateConformalExtension}
\(\mathfrak{H} u \in W^{1, n + 1} (\Bset^{n + 1})\)
and 
\[
  \int_{\Hset^{n + 1}} \abs{D (\mathfrak{H} u)}^{n + 1} \dif \mu_{\Hset^{n + 1}}
  \le C \int_{\Sset^{n}}\int_{\Sset^{n}} \frac{\abs{u (y) - u (z)}^{n + 1}}{\abs{y - z}^{2 n}}
  \dif y \dif z\ ,
\]
where the constant \(C\) only depends on the dimension \(n\);
\item \label{estimateConformalNontangential}if \(y \in \Sset^{n}\) and 
\[
\int_{\Sset^{n}} \frac{\abs{u (z) - u (y)}^{n + 1}}{\abs{z - y}^{2 n}}
  \dif z < \infty\ ,
\]
then for every \(\alpha \in (0, \infty)\),
\[
 \lim_{\substack{x \to y\\ \abs{x - y} \le \alpha (1 - \abs{x})}} \mathfrak{H}u (x)
 = u (y)\ .
\]
\end{enumerate}
\end{proposition}

This proposition is classical (see for example \cite{Stein1970}*{proposition V.7'}). Part \eqref{estimateConformalNontangential} is a nontangential convergence result. We prove it in this special case to show how the M\"obius covariance simplifies some parts of the argument.

\begin{proof}[Proof of proposition~\ref{propositionConformalExtension}]
We have for every \(x \in \Bset^{n + 1}\),
\[
  \mathfrak{H} u (x) - \mathfrak{H} u (0)
  = \frac{(1 - \abs{x}^2)^n}{\abs{\Sset^n}^2} \int_{\Sset^n} \int_{\Sset^n} \frac{ u(y) - u (z)}{\abs{y - x}^{2 n}} \dif y \dif z\ ,
\]
and therefore, if \(v \in \Rset^{n + 1}\),
\[
  D (\mathfrak{H}u) (0)[v]
  = \frac{2 n}{\abs{\Sset^n}^2}  \int_{\Sset^n}  \int_{\Sset^n} \bigl(u (z) - u (y)\bigr)\, y \cdot v \dif y\dif z\ ,
\]
and it follows thus that 
\[
 \abs{D (\mathfrak{H}u) (0)} \le \frac{2n}{\abs{\Sset^n}^2} \int_{\Sset^n} \int_{\Sset^n} \abs{u (y) - u (z)}\dif y
 \dif z\ .
\]

Next, by M\"obius covariance and since \(\abs{\cdot}_{\Hset^{n+1}} = \abs{\cdot}\) at the origin, we have 
\[
\begin{split}
 \abs{D (\mathfrak{H} u) (x)}_{\Hset^{n + 1}}^{n + 1}
 &= \abs{D (\mathfrak{H} u) (T_x^{-1} (0))}_{\Hset^{n + 1}}^{n + 1}
 = \abs{D (\mathfrak{H} (u \circ T_x^{-1})) (0)}^{n + 1}\\
 &\le \frac{2n}{\abs{\Sset^n}^2} \int_{\Sset^n}\int_{\Sset^n} \abs{(u \circ T_x^{-1}) (y) - (u \circ T_x^{-1}) (z)}^{n + 1} \dif y \dif z\\
 &= \frac{2n}{\abs{\Sset^n}^2} \int_{\Sset^n}\int_{\Sset^n} \frac{(1 - \abs{x}^2)^{2 n} \abs{u (y) - u (z)}^{n + 1}}{\abs{x - y}^{n} \abs{x - z}^{n}} \dif y \dif z\ .
\end{split}
\]
By integration over \(\Hset^{n + 1}\), this gives using \eqref{eqHyperbolicMeasure}
\[
\begin{split}
 \int_{\Hset^{n + 1}} \abs{D (\mathfrak{H} u)}_{\Hset^{n + 1}}^{n + 1}\dif \mu_{\Hset^{n + 1}}
 \le \frac{2^{n + 2} n}{\abs{\Sset^n}^2} \int_{\Bset^{n + 1}} \int_{\Sset^n}\int_{\Sset^n} \frac{(1 - \abs{x}^2)^{n - 1} \abs{u (y) - u (z)}^{n + 1}}{\abs{x - y}^{2 n} \abs{x - z}^{2 n}} \dif y \dif z \dif x \\
 =\frac{2^{n + 2} n}{\abs{\Sset^n}^2} \int_{\Sset^n}\int_{\Sset^n} \frac{\abs{u (y) - u (z)}^{n + 1}}{\abs{y - z}^{2 n}}
 \Bigl(\int_{\Bset^{n + 1}} \frac{(1 - \abs{x}^2)^{n - 1} \abs{y - z}^{2 n}}
 {\abs{x - y}^{2 n} \abs{x - z}^{2 n}} \dif x\Bigr)\dif y \dif z\ .
\end{split}
\]

We observe that, by a change of variable \(x = T (\Tilde{x})\), in view of the identities \eqref{eqMobiusGeneralDerivative} and \eqref{eqMobiusGeneralDifference}
\begin{multline*}
  \int_{\Bset^{n + 1}} \frac{(1 - \abs{x}^2)^{n - 1} \abs{T (y) - T (z)}^{2 n}}
 {\abs{x - T (y)}^{2 n} \abs{x - T (z)}^{2 n}} \dif x \\
 = \int_{\Bset^{n + 1}} \frac{(1 - \abs{T (\Tilde{x})}^2)^{n - 1} \abs{T (y) - T (z)}^{2 n}}
 {\abs{T (\Tilde{x}) - T (y)}^{2 n} \abs{T (\Tilde{x}) - T (z)}^{2 n}} \abs{D T (\Tilde{x})}^{n + 1}\dif \Tilde x\\
 =\int_{\Bset^{n + 1}} \frac{(1 - \abs{\Tilde{x}}^2)^{n - 1} \abs{y - z}^{2 n}}
 {\abs{\Tilde{x} - y}^{2 n} \abs{\Tilde{x} - z}^{2 n}} \dif \Tilde x\ .
\end{multline*}
Since we can choose a M\"obius transformation \(T \in \mobius (\Bset^{n + 1})\) such that the points \(T (y)\) and \(T (z)\) are antipodal, it follows that the integral does not depend on \(y\) and \(z\), 
and the required estimate of part \eqref{estimateConformalExtension} follows.

To prove part \eqref{estimateConformalNontangential}, we first note that for each \(x \in \Hset^{n + 1}\) and \(y \in \Sset^n\),
\[
 \mathfrak{H} u (x) - u (y) 
 = \frac{1}{\abs{\Sset^{n}}} \int_{\Sset^{n}} \frac{(1 - \abs{x}^2)^n}{\abs{z - x}^{2 n}} \bigl(u (z) - u (y)\bigr)\dif z\ ,
\]
and therefore by Jensen's inequality
\[
\begin{split}
 \abs{\mathfrak{H} u (x) - u (y)}^{n + 1}
 &\le \frac{1}{\abs{\Sset^n}} \int_{\Sset^{n}} \frac{(1 - \abs{x}^2)^n}{\abs{z - x}^{2 n}} \abs{u (z) - u (y)}^{n+1} \dif z.
\end{split}
\]
For \(x \in \Hset^{n + 1}\), \(y\in\mathbb S^n\) if \(|x-y|\le \alpha(1-\abs{x})\) and \(z \in \Sset^n \setminus \{y\}\), then \(1 - \abs{x} \le \abs{z  -x}\) and \(|x-y|\le \alpha\abs{z  -x}\). This gives the following:
\begin{multline*}
\frac{(1 - \abs{x}^2)^n \abs{u (z) - u (y)}^{n + 1} }{\abs{z - x}^{2 n}}= \frac{(1 - \abs{x}^2)^n\abs{z - y}^{2 n} }{\abs{z - x}^{2 n}} \frac{\abs{u (z) - u (y)}^{n + 1}}{\abs{z - y}^{2 n}}\\
\le (1 - \abs{x}^2)^n\frac{2^{2n - 1}(\abs{z - x}^{2 n} + \abs{y - x}^{2 n})}{\abs{z - x}^{2 n}} \frac{\abs{u (z) - u (y)}^{n + 1}}{\abs{z - y}^{2 n}}\\
\le 2^{2 n -1}\bigl(1 + \alpha^{2 n}\bigr)(1 - \abs{x}^2)^n
\frac{\abs{u (z) - u (y)}^{n + 1}}{\abs{z - y}^{2 n}}\ ,
\end{multline*}
and given the finiteness hypothesis of part \eqref{estimateConformalNontangential}, the conclusion follows from Lebesgue's dominated convergence theorem.
\end{proof}

\section{Good and bad points}

Given a map \(u : \Sset^n \to M \subset \Rset^\nu\), we cannot expect that the hyperharmonic extension \(\mathfrak{H} u\)
satisfies \(\mathfrak{H}u (\Hset^{n + 1}) \subset M\). 
Since the manifold \(M\) is compact, the nearest point projection \(\pi_M\) is 
well-defined and smooth in a neighborhood of radius \(\iota\) around \(M\). 
We would like to choose as an extension \(\pi_M \circ \mathfrak{H} (u)\). 
This will work when the original map \(u\) does not oscillate too much 
--- for example if \(\seminorm{u}_{W^{n/(n + 1), n + 1}}\) is small enough --- but shall not do it in the general case.

We shall consider that a given point \(x \in \Hset^{n + 1}\) is \emph{good} for the map \(\mathfrak{H} (u)\) if 
\[
 \dist \bigl(\mathfrak{H}u (x), u (\Sset^n)\bigr) < \iota\ .
\]
If \(u (\mathbb{S}^{n}) \subset M\), this will imply in particular that 
\[
  \dist \bigl(\mathfrak{H}u (x), u (\Sset^n)\bigr)
\le \dist (\mathfrak{H}u (x), M) < \iota\ .
\] 
Other points in \(\Hset^{n + 1}\) are \emph{bad} points.
We study in this section the structure of the good and bad sets.

\subsection{Compactness of the bad set}

We begin by showing that for every \(u \in W^{\frac{n}{n + 1}, n + 1} (\Sset^{n}; \Rset^\nu)\), the bad set of \(\mathfrak{H} u\) remains away from the boundary sphere \(\Sset^{n}\).

\begin{lemma}[Compactness of the bad set]
\label{lemmaGoodRegion}
If \(u \in W^{\frac{n}{n + 1}, n + 1}(\Sset^{n}; \Rset^\nu)\), then
\begin{enumerate}[(i)]
 \item \label{itemUniformEstimate} for each \(x \in \Hset^{n + 1}\),
\begin{equation*}
  \dist \bigl(\mathfrak{H} u (x), u (\Sset^{n})\bigr)
  \le \Bigl(\frac{4^{2n}}{\abs{S^{n}}^2}\int_{\Sset^n}\int_{\Sset^n} \frac{\abs{u (y) - u (z)}^{n + 1}}{\abs{y - z}^{2 n}}\dif y \dif z\Bigr)^\frac{1}{n + 1}.
\end{equation*}
\item \label{itemAsymptoticBehaviour} for every \(a \in \Hset^{n+1}\),
\[
  \lim_{d_{\Hset^{n + 1}} (a, x) \to \infty} \dist \bigl(\mathfrak{H} u (x), u (\Sset^n)\bigr) = 0\ .
\]
\end{enumerate}
\end{lemma}

This lemma is the consequence of the fact that the map \(\mathfrak{H} u\) is essentially an extension by convolution of the map \(u\), which has vanishing mean oscillation (VMO) (see \cite{BrezisNirenberg1995}*{(7)}).

Since the set of bad points
\[
 \bigl\{ x \in \Bset^{n + 1} \st \dist \bigl(\mathfrak{H} u (x), u (\Sset^n)\bigr) \ge \iota \bigr\}\ ,
\]
is a closed subset of the hyperbolic space \(\Hset^{n + 1}\) (by continuity of the function \(\mathfrak{H} u\) in the open ball \(\Bset^{n + 1}\)), lemma~\ref{lemmaGoodRegion} implies that this bad set is compact.

\begin{proof}[Proof of lemma~\ref{lemmaGoodRegion}]
For every \(x \in \Hset^{n + 1}\simeq \Bset^{n + 1}\), in view of \eqref{eqKernelConstant}, the definition of the hyperharmonic extension \eqref{equationConformalExtension} and Jensen's inequality, we have 
\begin{equation}
\label{eqDistanceHolder}
\begin{split}
  \dist \bigl(\mathfrak{H} u (x), u (\Sset^{n})\bigr)
  &\le \frac{1}{\abs{\Sset^{n}}} \int_{\Sset^{n}} \frac{(1 - \abs{x}^2)^n\abs{\mathfrak{H} u (x) - u (z)}}{\abs{x - z}^{2 n}}\dif z\\
  &\le \frac{1}{\abs{\Sset^{n}}^2} \int_{\Sset^n} \int_{\Sset^{n}} \frac{(1 - \abs{x}^2)^{2 n}\abs{u (y) - u (z)}}{\abs{x - y}^{2 n} \abs{x - z}^{2 n}} \dif y \dif z\\
  &\le \Bigl( \frac{1}{\abs{\Sset^{n}}^2} \int_{\Sset^n} \int_{\Sset^{n}} \frac{(1 - \abs{x}^2)^{2 n}\abs{u (y) - u (z)}^{n + 1}}{\abs{x - y}^{2 n} \abs{x - z}^{2 n}} \dif y \dif z\Bigr)^\frac{1}{n + 1}\ .
\end{split}
\end{equation}
We observe that, if \(\abs{x - y} \le \abs{x - z}\), by the triangle inequality
\[
 (1 - \abs{x}^2)\, \abs{y - z}
 \le (1 + \abs{x}) \,\abs{x - y}\, \bigl(\abs{x - y} + \abs{x - z}\bigr)
 \le 4 \abs{x - y}\, \abs{x - z}\ .
\]
Since the above left-hand side and the right-hand side are invariant under permutation of \(y\) and \(z\), we have for every \(x \in \Hset^{n + 1} \simeq \Bset^{n + 1}\) and \(y, z \in \Sset^n\),
\begin{equation}
\label{eqLargeBallsBound}
\frac{(1 - \abs{x}^2)^{2 n}\abs{u (y) - u (z)}^{n + 1}}{\abs{x - y}^{2 n} \abs{x - z}^{2 n}}
\le 4^{2 n}\frac{\abs{u (y) - u (z)}^{n + 1}}{\abs{y - z}^{2 n}}\ .
\end{equation}
The assertion \eqref{itemUniformEstimate} follows then from the inequalities \eqref{eqDistanceHolder} and \eqref{eqLargeBallsBound}.

In order to prove \eqref{itemAsymptoticBehaviour},
since the closed unit Euclidean ball \(\overline{\Bset^{n + 1}}\) is compact, it suffices to prove that for every sequence
\((x_k)_{k \in \Nset}\) of points in the open ball \(\Bset^{n + 1}\) converging to an arbitrary point \(\Bar{x} \in \Sset^{n}\), one has 
\begin{equation}
\label{eqLargeBallsSequence}
  \lim_{k \to \infty} \dist \bigl(\mathfrak{H} u (x_k), u (\Sset^n)\bigr) = 0\ .
\end{equation}
We note that for every \(y, z \in \Sset^{n} \setminus \{\Bar{x}\}\), 
\begin{equation}
\label{eqLargeBallsConvergence}
 \lim_{k \to \infty} \frac{(1 - \abs{x_k}^2)^{2 n} \abs{u (y) - u (z)}^{n + 1}}{\abs{x_k - y}^{2 n} \abs{x_k - z}^{2 n}} = 0\ .
\end{equation}
In view of the convergence \eqref{eqLargeBallsConvergence}, of the bound \eqref{eqLargeBallsBound} and of the assumption \(u \in W^{\frac{n}{n + 1}, n + 1} (\Sset^{n}; \Rset^\nu)\), by Lebesgue's dominated convergence  we obtain
\[
 \lim_{k \to \infty} \int_{\Sset^n} \int_{\Sset^{n}} \frac{(1 - \abs{x_k}^2)^{2 n}\abs{u (y) - u (z)}^{n + 1}}{\abs{x_k - y}^{2 n} \abs{x_k - z}^{2 n}} \dif y \dif z = 0\ ,
\]
form which the convergence
\eqref{eqLargeBallsSequence} and \eqref{itemAsymptoticBehaviour} follow.
\end{proof} 

\begin{remark}
\label{remarkRetraction}
If \(\pi_M\) is a Lipschitz retraction from a neighborhood of size \(\iota\) of \(M\) to \(M\) and if the Gagliardo seminorm of the map \(u\) is small, that is,
\[
 \frac{4^{2n}}{\abs{S^{n}}^2}\int_{\Sset^n}\int_{\Sset^n} \frac{\abs{u (y) - u (z)}^{n + 1}}{\abs{y - z}^{2 n}}\dif y \dif z < \iota^{n + 1}\ ,
\]
where the constant is coming from lemma~\ref{lemmaGoodRegion} \eqref{itemUniformEstimate}, then an extension \(U\) can be constructed  by setting 
\(U = \pi_M \circ \mathfrak{H}u\). One has then \(U \in W^{1, n + 1}(\Bset^{n + 1}; M)\) and the norm of the extension is controlled linearly:
\[
  \int_{\Bset^{n + 1}} \abs{D U}^{n + 1} \le C' \int_{\Sset^n}\int_{\Sset^n} \frac{\abs{u (y) - u (z)}^{n + 1}}{\abs{y - z}^{2 n}} \dif y \dif z\ ,
\]
for some constant \(C'=CL^{n+1}> 0\) where \(C\) is the constant of proposition~\ref{propositionConformalExtension} and \(L\) is the Lipschitz constant of \(\pi_M\) (see \cite{PetracheRiviere2015}*{proposition 2.6 and theorem 4.4}).
\end{remark}

\subsection{Estimate on the density of spheres in the good set}

We are interested in controlling the norm of the extension in terms of the norm of the original function.
In this respect lemma~\ref{lemmaGoodRegion} is not good enough: the estimate \eqref{itemUniformEstimate} cannot be used unless the Gagliardo fractional seminorm \(\seminorm{u}_{W^{\frac{n}{n + 1}, n + 1}}\) is small enough; 
the bad region is bounded with an estimate obtained by Lebesgue's dominated convergence that depends not only on the Gagliardo fractional norm \(\seminorm{u}_{W^{\frac{n}{n + 1}, n + 1}}\), 
but also on the modulus of integrability of the integrand in the latter seminorm (see remark \ref{remarkModulusOfIntegrability}).

Several quantitative estimates on the bad set are already known.
First, the measure of the bad set can be bounded by the Hardy inequality:
\[
\begin{split}
  \mu_{\Hset^{n + 1}} \bigl(\bigl\{x \in \Hset^{n + 1} \st \dist \bigl(U (x), u (\Sset^n)\bigr) \ge \iota\bigr\}\bigr)
  &\le \frac{1}{\iota^{n + 1}}
  \int_{\Hset^{n + 1}} \dist (U, u (\Sset^n))^{n + 1} \dif \mu_{\Hset^{n + 1}}\\
  &\le \frac{C}{\iota^{n + 1}} \int_{\Hset^{n + 1}} \abs{D U}_{\Hset^{n + 1}} \dif \mu_{\Hset^{n + 1}}.
\end{split}
\]
The \(W^{1, n + 1}\), or conformal, capacity of the bad set relative to the ball is controlled naturally
\[
  \operatorname{cap}_{W^{1, n + 1}} \bigl(\{x \in \Bset^{n + 1} \st \dist \bigl(U (x), u (\Sset^n)\bigr)\ge \iota  \}, \Bset^{n + 1}\bigr)
  \le \frac{1}{\iota^{n + 1}} \int_{\Bset^{n + 1}} \abs{D U}^{n + 1};
\]
this gives a control on the hyperbolic diameter of the \emph{connected components} of the bad set \cite{Mostow1968}*{theorem 8.2}.
One has \cite{BourgainBrezisMironescu2005}*{proof of lemma 1.3}: 
\begin{equation}
\label{ineqBourgainBrezisMironescu}
 \int_{\Sset^{n}} \sup_{r \in (0, 1)} \frac{\abs{U (r x) - U (x)}^{n + 1}}{(\ln 1/r)^n} \dif x
 \le C \int_{\Bset^{n + 1}} \abs{D U}^{n + 1},
\end{equation}
from which an estimate about an average diameter of the bad set follows
\[
 \int_{\Sset^{n}} \sup_{r \in (0, 1)} \bigl\{(\ln 1/r)^{-n} \st \dist \bigl(U (rx), u(\Sset^n)\bigr) \ge \iota \bigr\}\dif x
 \le \frac{C}{\iota^{n + 1}} \int_{\Bset^{n + 1}} \abs{D U}^{n + 1}.
\]

We give a new quantitative estimate on the fraction of the spheres of radius between \(0\) and \(\rho\) which are bad. 
This estimate will only depend on the parameter \(\iota\) and on the Gagliardo seminorm \(\seminorm{u}_{W^{\frac{n}{n + 1}, n + 1}}\). 
It is a crucial new ingredient in our construction of the singular extension.

\begin{proposition}[Estimate on the density of spheres in the good set]
\label{propositionGoodRadiusFractional}
There exists \(C > 0\) such that for every \(\rho > 0\) and for every \(a\in\Hset^{n + 1}\)
we have
\begin{multline*}
 \frac{1}{\rho}\int_0^\rho  \norm{\dist (\mathfrak{H} u, u (\Sset^n))}_{L^\infty (\partial\Bset^{\Hset^{n + 1}}_r(a))}\dif r\\
 \le \frac{C}{1 + \rho^\frac{1}{n + 1}} \left(\int_{\Sset^n} \int_{\Sset^n} \frac{\abs{u (y) - u (z)}^{n + 1}}{\abs{y - z}^{2 n}}\dif y \dif z \right)^\frac{1}{n + 1} .
\end{multline*}
\end{proposition}

By the classical Chebyshev inequality, proposition~\ref{propositionGoodRadiusFractional} implies that the proportion of bad spheres of whose radius is between \(0\) and \(\rho\) is bounded from above by \(\seminorm{u}_{W^{\frac{n}{n + 1}, n + 1}}/(\iota \rho^{1/(n + 1)})\).

When the bad set is connected, proposition~\ref{propositionGoodRadiusFractional} gives an estimate on the hyperbolic diameter of the bad set.
Compared to \eqref{ineqBourgainBrezisMironescu}, the radial and spherical coordinates are reversed in the supremum and in the integral in proposition~\ref{propositionGoodRadiusFractional}.

\begin{proof}[Proof of proposition~\ref{propositionGoodRadiusFractional}]
When \(\rho \le 1\), the estimate follows from lemma~\ref{lemmaGoodRegion} \eqref{itemUniformEstimate}. 
We assume thus that \(\rho \ge 1\).

Since \(T_a (a) = 0\) and since \(T_a\) is a hyperbolic isometry we have \(T_a(\partial \mathbb B^{\mathbb H^{n+1}}_r(a)) =\partial \mathbb B^{\mathbb H^{n+1}}_r(0)\). 
In view of the covariance of the hyperharmonic extension \(\mathfrak{H} u\) under such maps \(T_a\)  (see proposition~\ref{propositionConformalExtension}), we observe that 
\[
\begin{split}
 \norm{\dist (\mathfrak{H} (u), u (\Sset^n))}_{L^\infty (\partial\Bset^{\Hset^{n + 1}}_r(a))}
 &=\norm{\dist ((\mathfrak{H} u) \circ T_a^{-1}, u (\Sset^n))}_{L^\infty (\partial\Bset^{\Hset^{n + 1}}_r(0))}\\
 &=\norm{\dist (\mathfrak{H} (u \circ T_a^{-1}), u (\Sset^n))}_{L^\infty (\partial\Bset^{\Hset^{n + 1}}_r(0))}\ .
\end{split}
\]
By conformal invariance of the Gagliardo fractional seminorm 
\(\seminorm{u}_{W^{n/(n + 1), n + 1}}\) (see \eqref{eqGagliardoConformal}),
we can thus also assume that \(a = 0\).

By elementary integral inequalities and the definition of the hyperharmonic extension \(\mathfrak{H}u\) \eqref{equationConformalExtension}, we have successively, for each \(x \in \Hset^{n + 1} \simeq \Bset^{n + 1}\),
\begin{equation}
\label{ineqPointwiseDistanceOscillationComposite}
\begin{split}
 \dist \bigl(\mathfrak{H}u (x), u (\Sset^{n})\bigr) &\le \frac{1}{\abs{\Sset^n}}
 \int_{\Sset^{n}} \abs{\mathfrak{H} u (x) - u (T_{x}^{-1} (z))} \dif z\\
 &= \frac{1}{\abs{\Sset^{n}}} \int_{\Sset^n} \Bigabs{\frac{1}{\abs{\Sset^{n}}}\int_{\Sset^n} u (T_{x}^{-1} (y)) - u (T_{x}^{-1} (z))\dif y} \dif z\\
 &\le  \frac{1}{\abs{\Sset^{n}}^2}\int_{\Sset^n} \int_{\Sset^n} \abs{u (T_{x}^{-1} (y)) - u (T_{x}^{-1} (z))} \dif z \dif y\ .
\end{split}
\end{equation}
By a change of variable, in view of the derivative formula for M\"obius transformations \eqref{eqMobiusDerivative} we have
\[
  \int_{\Sset^n} \int_{\Sset^n} \abs{u (T_{x}^{-1} (y)) - u (T_{x}^{-1} (z))} \dif z \dif y
  =  \int_{\Sset^n} \int_{\Sset^n} \frac{(1 - \abs{x}^2)^{2n} \abs{u (y) - u (z)}}{\abs{x - y}^{2 n}\abs{x - z}^{2 n}}\dif y \dif z\ ,
\]
and therefore, for each \(x \in \Hset^{n + 1} \simeq \Bset^{n + 1}\),
\[
\begin{split}
  \frac{\dist \bigl(\mathfrak{H}u (x), u (\Sset^{n})\bigr)}{1 - \abs{x}^2} &\le\frac{1}{\abs{\Sset^{n}}^{2}} \int_{\Sset^n} \int_{\Sset^n} \frac{(1 - \abs{x}^2)^{2n - 1} \abs{u (y) - u (z)}}{\abs{x - y}^{2 n}\abs{x - z}^{2 n}}\dif y \dif z\\
  &= \frac{1}{\abs{\Sset^{n}}^{2}}\int_{\Sset^n} \int_{\Sset^n} \frac{(1 - \abs{x}^2)^{2n - 1} \abs{y - z}^\frac{2 n}{n + 1}}{\abs{x - y}^{2 n} \abs{x - z}^{2 n}} \frac{\abs{u (y) - u (z)}}{\abs{y - z}^\frac{2 n}{n + 1}} \dif y \dif z\ .
\end{split}
\]

We observe that for every \(x \in \Bset^{n + 1}\) and every \(y, z \in \Sset^n\),
\[
(1 - \abs{x}^2) = \abs{y}^2 - \abs{x}^2 = \abs{z}^2 - \abs{x}^2 \le 2 \min (\abs{x - y}, \abs{x - z})
\]
and that, by the triangle inequality
\[
  \abs{y - z} \le|x-y|+\abs{x-z}\le 2 \max (\abs{x - y}, \abs{x - z})\ .
\]
We have thus
\begin{equation}
\label{ineqGoodSphereIneqKernel}
\begin{split}
\frac{(1 - \abs{x}^2)^{2n - 1} \abs{y - z}^\frac{2 n}{n + 1}}{\abs{x - y}^{2 n} \abs{x - z}^{2 n}}\hspace{-2cm}\ &\\
&\le  \frac{2^{2n - 1 + \frac{2n}{n + 1}}\min (\abs{x - y}, \abs{x - z})^{2n - 1} \max (\abs{x - y}, \abs{x - z})^\frac{2 n}{n + 1}}{\abs{x - y}^{2 n} \abs{x - z}^{2 n}}\\
&= \frac{2^{2n - 1 + \frac{2n}{n + 1}}}{\min (\abs{x - y}, \abs{x - z}) \max (\abs{x - y}, \abs{x - z})^\frac{2 n^2}{n + 1}}\\
&\le \frac{2^{4n - 2}}{\abs{x - y}\, \abs{x - z}\, (\abs{x - y} + \abs{x - z})^{\frac{2 n^2}{n + 1} - 1}}\ ,
\end{split}
\end{equation}
and therefore, for each \(x \in \Hset^{n + 1} \simeq \Bset^{n + 1}\)
\[
  \frac{\dist \bigl(\mathfrak{H} u (x), u (\Sset^{n})\bigr)}{1 - \abs{x}^2}
  \le \frac{2^{4n - 1}}{\abs{\Sset^{n}}^2}\int_{\Sset^n} \int_{\Sset^n}\frac{f (y, z)}{\abs{x - y}\, \abs{x - z} (\abs{x - y} + \abs{x - z})^{\frac{2 n^2}{n + 1} - 1}} \dif y \dif z\ ,
\]
where the function \(f : \Sset^{n} \times \Sset^{n} \to \Rset\) is defined for each \(y, z \in \Sset^n\) by 
\[
 f (y, z) = \frac{\abs{u (y) - u (z)}}{\abs{y - z}^\frac{2 n}{n + 1}}\ .
\]

We fix \(e \in \Sset^n\) and we define then the symmetric decreasing rearrangement of \(f\) with respect to the first variable to be the function \(f^* : \Sset^n \times \Sset^n \to \Rset\) such that for each \(\lambda > 0\) and \(z \in \Sset^n\), 
\(f^*(\cdot, z)^{-1} ((\lambda, \infty))\) is a geodesic ball centered at \(e\) that has the same measure in \(\Sset^{n}\) as \(f (\cdot, z)^{-1} ((\lambda, \infty))\).
By the classical Hardy--Littlewood rearrangement inequality \cite{CroweZweibelRosenbloom1986} (see also \cite{LiebLoss}*{theorem 3.4}), we have, for every \(z \in \Sset^n\),
\begin{multline}
\label{ineqRearr1}
 \int_{\Sset^n} \frac{f (y, z)}{\abs{x - y}\, \abs{x - z}\, (\abs{x - y} + \abs{x - z})^{\frac{2 n^2}{n + 1} - 1}} \dif y\\
 \le \int_{\Sset^n} \frac{f^* (y, z)}{\abs{y - \abs{x}e}\, \abs{x - z}\, (\abs{y - \abs{x} e} + \abs{x - z})^{\frac{2 n^2}{n + 1}-1}} \dif y\ .
\end{multline}
We also define the the symmetric decreasing rearrangement of \(f^*\) with respect to the second variable to be the function \(f^{**} : \Sset^n \times \Sset^n \to \Rset\) such that for each \(\lambda > 0\) and \(y \in \Sset^n\), 
\(f^{**}(y, \cdot)^{-1} ((\lambda, \infty))\) is a geodesic ball centered at \(e\) whose measure in \(\Sset^{n}\) is the same as the measure of \(f^* (y, \cdot)^{-1} ((\lambda, \infty))\).
By the Hardy--Littlewood rearrangement inequality again, we have, for every \(y \in \Sset^n\),
\begin{multline}
\label{ineqRearr2}
\int_{\Sset^n} \frac{f^* (y, z)}{\abs{y - \abs{x}e} \abs{x - z}\, (\abs{y - \abs{x} e} + \abs{x - z})^{\frac{2 n^2}{n + 1} - 1}} \dif z\\
\le \int_{\Sset^n} \frac{f^{**} (y, z)}{\abs{y - \abs{x}e}\, \abs{z - \abs{x}e} \, (\abs{y - \abs{x} e} + \abs{z - \abs{x}e})^{\frac{2 n^2}{n + 1} - 1}} \dif z\ .
\end{multline}
The combination of the rearrangement inequalities \eqref{ineqRearr1} and \eqref{ineqRearr2} 
implies that for every \(x \in \Hset^{n + 1} \simeq \Bset^{n + 1}\),
\begin{multline*}
  \frac{\dist \bigl(\mathfrak{H}u (x), u (\Sset^{n})\bigr)}{1 - \abs{x}^2}\\
  \le \frac{2^{4n - 1}}{\abs{\Sset^{n}}^2}\int_{\Sset^n} \int_{\Sset^n} \frac{f^{**} (y, z)}{\abs{y - \abs{x}e}\, \abs{z - \abs{x}e}\,  (\abs{y - \abs{x} e} + \abs{z - \abs{x}e})^{\frac{2 n^2}{n + 1} - 1}}\dif y \dif z\ .
\end{multline*}
We now observe that for each \(y \in \Sset^n\) and \(x \in \Hset^{n + 1} \simeq \Bset^{n + 1}\),
\begin{equation}
\label{ineqTriple}
  3 \abs{y - \abs{x} e} \ge (\abs{y - e} - \abs{e - \abs{x}e}) + 2(\abs{y} - \abs{\abs{x}e})
= \abs{y - e} + 1 - \abs{x}\ ,
\end{equation}
and thus for every \(x \in \Hset^{n + 1} \simeq \Bset^{n + 1}\)
\begin{multline}
\label{eqtsrn}
\frac{\dist (\mathfrak{H}u (x), u (\Sset^{n}))}{1 - \abs{x}^2}\\
  \le C_1 \int_{\Sset^n} \int_{\Sset^n} \frac{f^{**} (x, y)}{(1 - \abs{x} + \abs{e - y})(1 - \abs{x}+ \abs{e - z} )(1 - \abs{x} + \abs{e - y} + \abs{e - z})^{\frac{2 n^2}{n + 1} - 1}}\dif y \dif z\ .
\end{multline}

If we set \(R = \tanh \frac{\rho}{2}\), so that \(\Bset^{\Hset^{n + 1}}_R = \Bset^{n + 1}_\rho\), we have by definition of the Poincar\'e metric and by \eqref{eqtsrn}, for every \(x \in \Hset^{n + 1} \simeq \Bset^{n + 1}\)
\begin{multline*}
\int_0^\rho \norm{\dist (\mathfrak{H} u, u (\Sset^n))}_{L^\infty (\partial\Bset^{\Hset^{n + 1}}_r(a))} \dif r
= \int_0^R
\frac{\norm{\dist (\mathfrak{H} u, u (\Sset^n))}_{L^\infty (\partial\Bset_r(0))}}{1 - r^2}\dif r \\
\le C_1 \int_0^R \int_{\Sset^n} \int_{\Sset^n} \frac{f^{**} (x, y)}{(1 - r + \abs{e - y})\, (1 - r + \abs{e - z}) (1 - r + \abs{e - y}+ \abs{e - z})^{\frac{2 n^2}{n + 1} - 1}}\dif y \dif z \dif r\ .
\end{multline*}
By the H\"older inequality on \((0, R)\), we have for each \(y, z \in \Sset^{n}\), 
\begin{multline*}
\int_0^R \frac{1}{(1 - r + \abs{e - y})\, (1 - r + \abs{e - z}) (1 - r + \abs{e - y}+ \abs{e - z})^{\frac{2 n^2}{n + 1} - 1}} \dif r\\
\le \Bigl(\int_0^R \frac{1}{(1 - r + \abs{e - y})^{n + 1}} \dif r\Bigr)^\frac{1}{n + 1}
\Bigl(\int_0^R \frac{1}{(1 - r + \abs{e - z})^{n + 1}} \dif r\Bigr)^\frac{1}{n + 1}\\
\shoveright{
\times \Bigl(\int_0^R \frac{1}{(1 - r + \abs{e - y}+ \abs{e - z})^{2n + 1}} \dif r \Bigr)^\frac{n - 1}{n + 1}}\\
\le \frac{C_2}{(1 - R + \abs{e - y})^\frac{n}{n + 1}(1 - R + \abs{e - z})^\frac{n}{n + 1}(1 - R + \abs{e - y} + \abs{e - z})^\frac{2n (n - 1)}{n + 1}}\ .
\end{multline*}
By the H\"older inequality on \(\Sset^n \times \Sset^n\) we deduce
\begin{multline}
\label{ineqGoodSpheresFinalHolder}
\int_0^\rho \norm{\dist (\mathfrak{H} u, u (\Sset^n))}_{L^\infty (\partial\Bset^{\Hset^{n + 1}}_r(a))} \dif r\le C_3 \Bigl(\int_{\Sset^{n}} \int_{\Sset^{n}} \abs{f^{**}}^{n + 1}\Bigr)^\frac{1}{n + 1}\\
\times \Bigl(\int_{\Sset^{n}} \int_{\Sset^{n}}\frac{1}{(1 - R + \abs{e - y})(1 - R + \abs{e - z})(1 - R + \abs{e - y} + \abs{e - z})^{2 n - 2} }\dif y \dif z\Bigr)^\frac{n}{n + 1}.
\end{multline}
By the Cavalieri principle and by definition of \(f\), we obtain 
\[
\begin{split}
\int_{\Sset^{n}\times \Sset^{n}} \abs{f^{**}}^{n + 1}&=\int_{\Sset^{n}\times \Sset^{n}} \abs{f^{*}}^{n + 1}=\int_{\Sset^{n} \times \Sset^{n}} \abs{f}^{n + 1}\\
&=\int_{\Sset^{n}} \int_{\Sset^{n}} \frac{\abs{u (y) - u (z)}^{n + 1}}{\abs{y - z}^{2 n}}\dif y \dif z\ .
\end{split}
\]
The second integral in \eqref{ineqGoodSpheresFinalHolder} is bounded by
\begin{multline*}
  C_4 \int_{\Sset^n} \int_{\abs{z - e} \ge \abs{y - e}} \frac{1}{(1 - R + \abs{e - y})(1 - R + \abs{e - y} + \abs{e - z})^{2 n - 1} }\dif y \dif z\\
  \le C_5 \int_{\Sset^n} \frac{1}{(1 - R + \abs{e - y})^n} \dif y 
  \le C_6 \Bigl(\ln \frac{1}{1 - R} + 1\Bigr)\\
  =C_6 \Bigl(\ln \frac{e^{\rho} + 1}{2} + 1\Bigr)
  \le C_6 (\rho + 1)\ ,
\end{multline*}
since \(R = \tanh \frac{\rho}{2}\). 
The conclusion follows because \(\rho^{-1}(\rho +1)^{\frac{n}{n+1}}\simeq (1+\rho^{\frac{1}{n+1}})^{-1}\) for \(\rho \ge 1\).
\end{proof}

\begin{remark}
\label{remarkModulusOfIntegrability}
The strategy of proof of proposition~\ref{propositionGoodRadiusFractional} also works on the Euclidean space and allows to prove that if \(v \in W^{\frac{n}{n + 1}, n + 1} (\Rset^n)\) and \(\eta_r (x) = \eta (x / r)/r^n\) with \(\rho \in C_c (\Rset^n)\), then 
\begin{multline}
 \frac{1}{\ln \frac{R}{\rho}} \int_\rho^R \norm{\dist (\eta_r \ast v, v (\Rset^n))}_{L^\infty} \frac{\dif r}{r}\\
 \le \frac{C}{1 + \bigl(\ln \frac{R}{\rho}\bigr)^\frac{1}{n + 1}} \Bigl(\int_{\Rset^n} \int_{\Rset^n} \frac{\abs{v (y) - v (z)}^{n + 1}}{\abs{y - z}^{2n}}\dif y \dif z \Bigr)^\frac{1}{n + 1}\ .
\end{multline}
\end{remark}

\section{Concentration of singularities}
\label{sectionHomogeneous}

In the previous section, we have established the set of bad points is compact and that the set of good points contains many hyperbolic spheres of controlled radius. 
We would now like to propagate the good values from the boundary of hyperbolic balls that contain bad points.

\subsection{Extension on a hyperbolic ball}
Our first tool is an estimate of the radial extension: starting from a map in \(W^{1, n + 1} (\partial \Bset^{\Hset^{n + 1}}_\rho(a), \Rset^\nu)\), we extend it by homogeneity inside.
The exponent is \(n + 1\) critical: this extension is not in the space 
\(W^{1, n + 1} (\Bset^{\Hset^{n + 1}}_\rho(a), \Rset^\nu)\) 
because of this singularity at \(a\), 
but this singularity is however compatible with a Marcinkiewicz weak \(L^{n + 1}\) estimate on the derivative.

\begin{lemma}[Radial extension]
\label{lemmaHomogeneousExtension}
Let \(n \ge 1\), \(a \in \Hset^{n + 1}\), \(\rho > 0\) and \(\pi_{a, \rho}\) be the nearest point projection from \(\Bset^{\Hset^{n + 1}}_\rho (a) \setminus \{a\}\) to \(\partial \Bset^{\Hset^{n + 1}}_\rho (a)\).
If \(u  \in W^{1, n + 1} (\partial \Bset^{\Hset^{n + 1}}_\rho(a); \Rset^\nu)\), then 
\begin{enumerate}[(i)]
 \item \label{lemmaHomogeneousWeaklyDifferentiable} \(u \circ \pi_{a, \rho}\) is weakly differentiable in \(\Bset^{\Hset^{n + 1}}_\rho (a)\),
\vspace{1em}
\item \label{lemmaHomogeneousWeakEstimate}for every \(\lambda > 0\),
\[
\lambda^{n + 1} \mu_{\Hset^{n + 1}}\bigl(\bigl\{x \in \Bset^{\Hset^{n + 1}}_\rho (a) \st \abs{D( u \circ \pi_{a, \rho})(x)} > \lambda\bigr\} \bigr)
\le \frac{\sinh (\rho)}{n+1} \int_{\partial \Bset^{\Hset^{n + 1}}_\rho}  \abs{D u}^{n + 1}\ ,
\]
\item \label{lemmaHomogeneousNonSingular}
for every \(\delta > 0\), \( u \circ \pi_{a, \rho} \in W^{1, n + 1} (\Bset^{\Hset^{n + 1}}_\rho \setminus \Bset^{\Hset^{n + 1}}_\delta)\) and
\[
  \int_{\Bset^{\Hset^{n + 1}}_\rho \setminus \Bset^{\Hset^{n + 1}}_\delta}
    \abs{D (u \circ \pi_{a, \rho})}^{n+1}\dif \mu_{\Hset^{n + 1}} = \sinh(\rho)\ln\frac{\tanh\left(\frac{\rho}{2}\right)}{\tanh\left(\frac{\delta}{2}\right)}\int_{\partial \Bset^{\Hset^{n + 1}}_\rho}  \abs{D u}^{n + 1}\ .
\]
\end{enumerate}
\end{lemma}

The nearest point projection \(\pi_{a, \rho}\) can be described geometrically as follows: starting from \(a\) one follows the geodesic passing through \(x \in \Bset^{\Hset^{n + 1}}_\rho (a)\); 
\(\pi_{a, \rho} (x)\) is the intersection point of this geodesic with \(\partial \Bset^{\Hset^{n + 1}}_\rho\).
Alternatively, in exponential coordinates around \(a\), 
\[
 \pi_{a, \rho} (x) = \frac{\rho}{\abs{x}} x\,.
\]

In particular, if the map \(u\) takes good values on \(\partial \Bset^{\Hset^{n + 1}}_\rho (a)\), then \(u \circ \pi_{a, \rho}\) takes good values on the whole \(\Bset^{\Hset^{n + 1}}_\rho (a)\). 
This will allow us to fill in balls whose boundary consists of good points.

The Euclidean counterparts of \eqref{lemmaHomogeneousWeakEstimate} and \eqref{lemmaHomogeneousNonSingular} are
\[
 \lambda^{n + 1} \bigabs{\bigl\{x \in \Bset^{n + 1} \st \abs{D (u \circ \pi_{a, \rho})(x)}> \lambda\bigr\}}
  \le \frac{\rho}{n + 1}  \int_{\partial \Bset_\rho} \abs{D u}^{n + 1},
\]
and
\[
  \int_{\Bset_\rho \setminus \Bset_\delta} \abs{D (u \circ \pi_{a, \rho})}^{n+1} \le \rho \ln \frac{\rho}{\delta} \int_{\partial \Bset_\rho} \abs{D u}^{n + 1}\ .
\]

\begin{proof}[Proof of Lemma~\ref{lemmaHomogeneousExtension}]
We first note that the map \(u \circ \pi_{a, \rho}\) is weakly differentiable in \(\Bset^{\Hset^{n + 1}}_\rho (a)\). 
In order to prove assertion \eqref{lemmaHomogeneousWeaklyDifferentiable}, we shall prove \eqref{lemmaHomogeneousWeakEstimate} in order to remove the singularity at \(a\).

Without loss of generality, we assume that \(a = 0\). 
We set \(R =  \tanh \frac{\rho}{2}\), so that in the Poincar\'e ball model of the hyperbolic space, \(\Bset^{\Hset^{n + 1}}_\rho (0) \simeq \Bset^{n + 1}_R (0)\).
We can then write explicitly for each \(x \in \Bset^{\Hset^{n + 1}}_\rho (0) \simeq \Bset^{n + 1}_R (0)\),
\[
 (u \circ \pi_{a, \rho})(x) = u \Bigl( \frac{R}{\vert x \vert}x\Bigr)\ ,
\]
and thus 
\begin{equation}
\label{eqSingularEstimatePoint}
 \abs{D (u \circ \pi_{a, \rho}) (x)} =  \frac{R}{\abs{x}} \abs{D u(\pi_{a, \rho} (x))}\ .
\end{equation}
For \(\lambda > 0\), we compute now by integration in spherical coordinates and by \eqref{eqSingularEstimatePoint},
\[
\begin{split}
  \mu_{\Hset^{n + 1}} \bigl(\bigl\{x \in &\Bset^{\Hset^{n + 1}}_\rho (0) \st \abs{D (u \circ \pi_{a, \rho})}_{\Hset^{n + 1}} (x) > \lambda \bigr\} \bigr)
 \\&= 2^{n + 1} \int_0^R \frac{\mathcal{H}^{n} \bigl(\bigl\{x \in \partial \Bset_r (0) \st (1 - r^2)\abs{D (u \circ \pi_{a, \rho})} (x) > 2\lambda \bigr\}\bigr)}{(1- r^2)^{n + 1}}\dif r\\
 &= 2^{n + 1} \int_0^R \frac{\mathcal{H}^{n} \bigl(\bigl\{x \in \partial \Bset_R (0) \st (1 - r^2)\abs{D u} (x) > 2\lambda r /R \bigr\}\bigr)r^n}{(1- r^2)^{n + 1}R^n}\dif r\\
 & \le 2^{n + 1} \int_0^1 \frac{\mathcal{H}^{n} \bigl(\bigl\{x \in \partial \Bset_R (0) \st (1 - r^2)\abs{D u} (x) > 2\lambda r /R \bigr\}\bigr)r^n}{(1- r^2)^{n + 1}R^n}\dif r\\
\end{split}
\]
By Fubini's theorem we obtain
\begin{equation}
\label{eqSingularEstimateFubini}
\begin{split}
  \mu_{\Hset^{n + 1}} \bigl(\bigl\{x \in \Bset^{\Hset^{n + 1}}_\rho (0) \st \abs{D (u \circ \pi_{a, \rho})}_{\Hset^{n + 1}} (x) > \lambda \bigr\} \bigr)
   \hspace{-16em}&\\ 
  &\le \frac{2^{n+1}}{R^n}\int_{\partial \Bset_R (0)}\int_{\frac{r}{1 - r^2} \le \frac{R \abs{D u (x)}}{2\lambda}}\frac{r^n}{(1-r^2)^{n+1}}\dif r\dif \mathcal H^n (x)\ .
\end{split}
\end{equation}
For every \(x \in \partial \Bset_R (0)\), we compute, by a change of variable \(s = \frac{r}{1 - r^2}\),
\begin{equation}
\label{eqSingularEstimaterIntegral}
\begin{split}
  \int_{\frac{r}{1 - r^2} \le \frac{R \abs{D u (x)}}{2\lambda} } \frac{r^n}{(1 - r^2)^{n + 1}}\dif r
  & \le \int_{\frac{r}{1 - r^2} \le \frac{R \abs{D u (x)}}{2\lambda} } \left(\frac{r}{1 - r^2}\right)^{n} \frac{1 + r^2}{(1 - r^2)^2} \dif r\\
  & = \int_0^{\frac{R \abs{D u (x)}}{2\lambda}} s^n \dif s = \frac{R^{n + 1} \abs{D u (x)}^{n + 1}}{(n + 1)\lambda^{n + 1}2^{n + 1}}\ .
\end{split}
\end{equation}
Therefore, by inserting the latest estimate \eqref{eqSingularEstimaterIntegral} into \eqref{eqSingularEstimateFubini}, we reach the inequality
\[
\begin{split}
  \lambda^{n + 1}\mu_{\Hset^{n + 1}} & \bigl(\bigl\{x \in \Bset^{\Hset^{n + 1}}_\rho \st \abs{D (u \circ \pi_{a, \rho})}_{\Hset^{n + 1}} (x) \ge \lambda \bigr\} \bigr)\\
  & \le \frac{R}{n+1} \int_{\partial \Bset_R} \abs{D u}^{n + 1} \dif \mathcal{H}^n\\
  & = \frac{2R}{(n+1)(1 - R^2)} \int_{\partial \Bset^{\Hset^{n + 1}}_\rho} \abs{D u}_{\mathbb H^{n + 1} }^{n + 1}= \frac{\sinh(\rho)}{n+1} \int_{\partial \Bset^{\Hset^{n + 1}}_\rho} \abs{D u}_{\mathbb H^{n + 1} }^{n + 1}\ ,
\end{split}
\]
and the inequality \eqref{lemmaHomogeneousWeakEstimate} is proved. 
Since \(n \ge 1\), this implies that \(\abs{D (u \circ \pi_{a, \rho})} \in L^1 (\Bset_R^{n+1} (0))\), and since points are removable for weakly differentiable maps starting from dimension \(2\) (see \cite{mazya}*{theorem 1.1.18}), \(u \circ \pi_{a, \rho}\) is weakly differentiable on the whole \(\Bset_R^{n+1} (0)\), proving \eqref{lemmaHomogeneousWeaklyDifferentiable}.

We finally prove \eqref{lemmaHomogeneousNonSingular}.
\[
\begin{split}
  \int_{\Bset^{\Hset^{n + 1}}_\rho (0) \setminus \Bset^{\Hset^{n + 1}}_\delta (0)} \abs{D (u \circ \pi_{a, \rho})}^{n + 1}_{\Hset^{n + 1}}\dif \mu_{\Hset^{n + 1}}
  & =\int_{\Bset_R (0) \setminus \Bset_d (0)} \abs{D (u \circ \pi_{a, \rho})}^{n + 1}\\
  &=\int_d^R \int_{\partial \Bset_R (0)} \abs{D u}^{n + 1} \frac{R}{r}\dif r\\
  &=\frac{2R}{1 - R^2} \ln \frac{R}{d} \int_{\partial \Bset^{\Hset^{n + 1}}_\rho} \abs{D u}^{n + 1}_{\Hset^{n + 1}} \dif \mathcal{H}^n\\
  & = \sinh (\rho) \ln \frac{\tanh \frac{\rho}{2}}{\tanh \frac{\delta}{2}} \int_{\partial \Bset^{\Hset^{n + 1}}_\rho} \abs{D u}^{n + 1}_{\Hset^{n + 1}}\ .
\end{split}
\]
\end{proof}

\begin{remark}
Lemma~\ref{lemmaHomogeneousExtension} allows to construct directly an extension 
in the space \(W^{1, (n + 1, \infty)}(\Bset^{n + 1};M)\), \emph{without control on the norm}. 
If \(\mathfrak{H}u\) is the hyperharmonic extension of \(u \in W^{n/(n + 1), n + 1} (\Sset^n; M)\) given by \eqref{equationConformalExtension}, then \(\mathfrak{H} u \in W^{1, n + 1} (\Bset^{n + 1}; \Rset^{\nu})\). 
By lemma~\ref{lemmaGoodRegion}, there exists \(\Bar{\rho}\) such that \(\dist (u, M) < \iota\) in \(\Hset^{n + 1} \setminus \Bset^{\Hset^{n + 1}}_{\Bar{\rho}} (0)\). 
We define then 
\[
 \Tilde{U} (x)
 =\begin{cases}
    \pi_M (u (\pi_{0, \Bar{\rho}} (x))) & \text{if \(x \in \Bset_{\Bar{\rho}} (0)\)}\ ,\\
    u (x) & \text{otherwise}\ ;
  \end{cases}
\]
since \(\Tilde{U} \in W^{1, (n + 1, \infty)} (\Bset^{n + 1}; \Rset^\nu)\) and since
\(\dist (\Tilde{U}, M) < \iota\), we can define \(U = \pi_M \circ \Tilde{U} \in  W^{1, (n + 1, \infty)} (\Bset^{n + 1}; M)\). 
However the estimate satisfied by the map \(U\) is 
\[
 \lambda^{n + 1}
 \bigabs{\bigl\{x \in \Bset^{n + 1} \st \abs{D U (x)} > \lambda \bigr\} }
 \le \frac{C}{(1 - \Bar{\rho})^\frac{1}{n + 1}} \int_{\Sset^{n + 1}} \int_{\Sset^{n + 1}}
 \frac{\abs{u (y) - u (z)}^{n + 1}}{\abs{y - z}^{2 n}} \dif y \dif z\ ,
\]
for every \(\lambda > 0\), 
where the radius \(\Bar{\rho}\) coming from lemma~\ref{lemmaGoodRegion} depends unfortunately on the modulus of integrability of the integrand of the Gagliardo seminorm.
\end{remark}

\subsection{Iterating radial extensions}
The idea to construct the extension is to apply the radial extension of lemma~\ref{lemmaHomogeneousExtension} on good spheres given by lemma~\ref{lemmaGoodRadius}.
In practice, this is more delicate. 
Indeed, good spheres \emph{may overlap} and the radial lemma~\ref{lemmaHomogeneousExtension} cannot then be applied on overlapping balls simultaneously. 
We shall apply thus lemma~\ref{lemmaHomogeneousExtension} sequentially.
A new problem arises, namely after lemma~\ref{lemmaHomogeneousExtension} has been applied at least once, the resulting map is not anymore in the Sobolev space \(W^{1, n + 1} (\Hset^{n + 1}, \Rset^\nu)\).
The next proposition applies lemma~\ref{lemmaHomogeneousExtension} in such a way as to avoid a given set of singularities.

\begin{proposition}[Iterating radial extensions]
\label{propositionOneBall}
Let \(U \in W^{1, 1}_\mathrm{loc} (\Hset^{n + 1}, \Rset^\nu)\), \(S \subset \Hset^{n + 1}\), \(N \subset \Rset^\nu\), \(\rho > 0\) and \(\delta \in (0, \rho)\).
If 
\begin{gather*}
 \# \bigl(S \cap \Bset^{\Hset^{n + 1}}_{3 \rho}(a)\bigr) \le \frac{\rho}{4 \delta}\ ,\\
  D U\in L^{n + 1} \bigl(\Bset^{\Hset^{n + 1}}_{2\rho} (a)\setminus \textstyle\bigcup_{b \in S} \Bset_{\delta}^{\Hset^{n + 1}} (b)\bigr)\ ,
\intertext{and}
  \bigabs{\bigl\{r \in (\rho, 2\rho) \st U (\partial\Bset^{\Hset^{n + 1}}_r(a)) \not\subseteq N \bigr\} }
  \le \frac{\rho}{4}\ ,
\end{gather*}
then there exists a function
\(V \in W^{1, 1}_{\mathrm{loc}} (\Hset^{n + 1};\Rset^\nu)\)
and a set \(R \subset S \cap \Bset_{2 \rho}^{\Hset^{n + 1}}(a)\)
such that
\begin{enumerate}[(i)]
 \item \label{eqPreservationSingularities} for each \(b \in R\), \(\Bset^{\Hset^{n + 1}}_\delta (b) \subset \Bset^{\Hset^{n + 1}}_{2 \rho} (a)\),\hphantom{x}
 \vspace{1em}
 \item \label{equationHomogeneousExtension2}%
~\hfill%
\(
  V \vert_{\Hset^{n + 1} \setminus  \Bset^{\Hset^{n + 1}}_{2 \rho} (a)} = U \vert_{\Hset^{n + 1} \setminus  \Bset^{\Hset^{n + 1}}_{2 \rho} (a)}\ ,
\)
~\hfill \hphantom{x}
\vspace{1em}

\item \label{equationHomogeneousExtension1} for each \(b \in S \setminus R\),
\begin{equation*}
   V \vert_{ \Bset^{\Hset^{n + 1}}_\delta (b)} = U \vert_{ \Bset^{\Hset^{n + 1}}_\delta (b)}\ .
\end{equation*}
\item\label{equationHomogeneousExtensionCloseToN}\hfill
 \(U^{-1} (N) \cup \Bset^{\Hset^{n + 1}}_\rho (a) \subseteq V^{-1} (N)\),\hfill \hphantom{x}
 \vspace{1em}
 \item \label{equationMarcinkiewiczBoundHomogeneousExtension} for each \(\lambda > 0\),
 \begin{multline*}
 \lambda^{n + 1} \mu_{\Hset^{n + 1}} \bigl( \bigl\{ x \in \Bset^{\Hset^{n + 1}}_\delta(a) \st 
 \abs{D V (x)}_{\Hset^{n + 1}} > \lambda \bigr\}\bigr)\\
 \le \frac{4\sinh(2 \rho)}{(n + 1) \rho}\int_{\Bset^{\Hset^{n + 1}}_{2 \rho}(a) \setminus \bigcup_{b \in S}\Bset_\delta^{\Hset^{n + 1}}(b)} \abs{D U}_{\Hset^{n + 1}}^{n + 1}\dif \mu_{\Hset^{n + 1}}\ ,
\end{multline*}
\item \label{equationLebesgueBoundHomogeneousExtension}
 \(\displaystyle\int_{\Bset^{\Hset^{n + 1}}_{2 \rho}(a) \setminus \bigcup_{b \in S\cup\{a\} \setminus R} \Bset^{\Hset^{n + 1}}_\delta (b)} \abs{D V}_{\Hset^{n + 1}}^{n + 1}\dif \mu_{\Hset^{n + 1}}\)\\
 \hspace*{\stretch{1}}\(\displaystyle
  \le  \Bigl(1 + \frac{4 \sinh(2 \rho)}{\rho}\ln\frac{\tanh (\rho)}{\tanh\left(\frac{\delta}{2}\right)}\Bigr)
  \int_{\Bset^{\Hset^{n + 1}}_{2 \rho}(a) \setminus \bigcup_{b \in S} \Bset^{\Hset^{n + 1}}_\delta (b)} \abs{D U}_{\Hset^{n + 1}}^{n + 1}\dif \mu_{\Hset^{n + 1}}\ ,\)
\end{enumerate}
\end{proposition}

As illustrated on figure \ref{figureIterateRadial}, the new map \(V\) 
coincides with the map \(U\) inside \(\mathbb B_{2\rho}^{\mathbb H^{n+1}}(a)\), 
and out of the preexisting singularities \(S\) of \(u\), those in \(S \setminus R\) are erased, whereas those in \(R\) are preserved  \eqref{equationHomogeneousExtension2}.
Moreover the map \(V\) has by \eqref{equationHomogeneousExtensionCloseToN} a larger good set in which it takes values in \(N\) than \(U\), at the price of a singularity created at the point \(a\), controlled in the critical Sobolev--Marcinkiewicz space \eqref{equationMarcinkiewiczBoundHomogeneousExtension}.
The Sobolev norm is controlled away from the new set of singularities \(R \cup \{a\}\) \eqref{equationLebesgueBoundHomogeneousExtension}.

\begin{figure}

\includegraphics{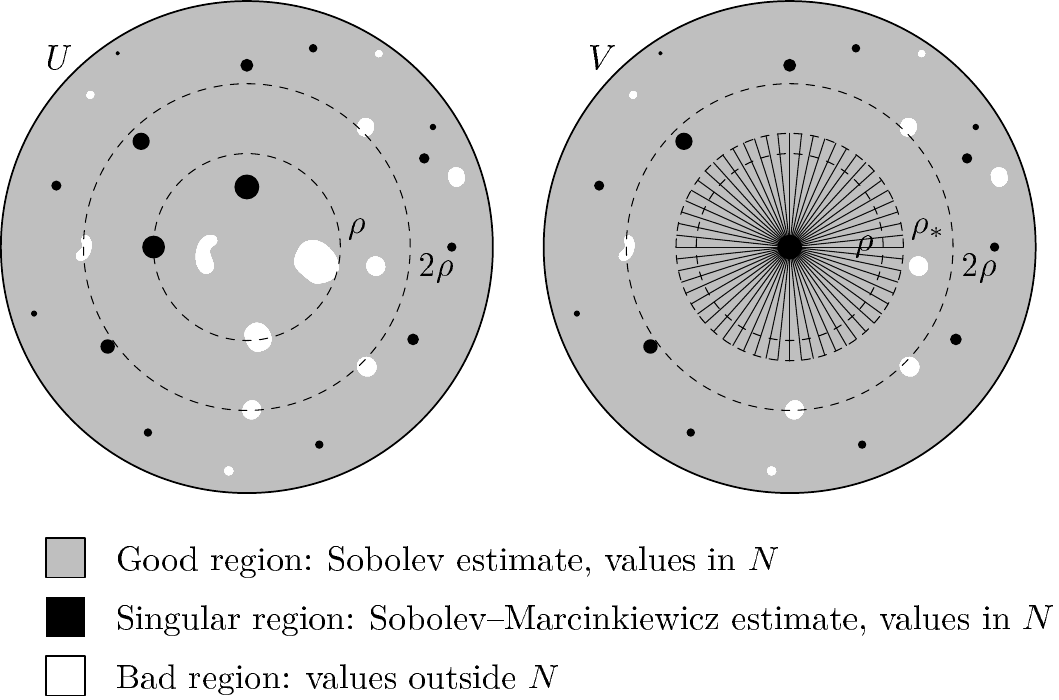}
  \caption{
  Proposition~\ref{propositionOneBall} transforms the map \(U\) on the left
  into the map \(V\) on the right by propagating by radial extension the values on a sphere of radius \(\rho_*\)
  in which \(U\) takes good values and is controlled in the Sobolev space,
  ensuring that the resulting function takes value into \(N\) inside the ball of radius \(\rho\) and 
  is controlled in Sobolev norm outside a small ball in which the singularity is also controlled.}
  \label{figureIterateRadial}
\end{figure}

\begin{proof}[Proof of proposition~\ref{propositionOneBall}]
We define the subset of real numbers
\begin{multline*}
    \mathcal R = \bigl\{ r \in (\rho, 2 \rho) \st U (\partial \Bset^{\Hset^{n + 1}}_r (a)) \subseteq N \\ \text{and for each \(b \in S\), } \partial\Bset^{\Hset^{n + 1}}_r (a) \cap \Bset_\delta (b) = \emptyset \bigr\}.
\end{multline*}
By the assumptions of the proposition, we have immediately
\[
  \abs{\mathcal R} \ge \rho - \frac{\rho}{4} - 2\delta\, \# S \ge \frac{\rho}{4}.
\]
By Fubini's theorem in the Sobolev space, for almost every \(r \in \mathcal R\), the map \(U\) coincides almost everywhere with its trace on the sphere \(\partial \Bset^{\Hset^{n + 1}}_r (a)\),
one has \(U \in W^{1, n + 1} (\partial \Bset^{\Hset^{n + 1}}_r (a))\),
\(D U\) is the weak derivative of \(U\) on \(\partial \Bset^{\Hset^{n + 1}}_r (a)\) and 
\[
  \int_{\mathcal R} \int_{\partial \Bset^{\Hset^{n + 1}}_r (a)} \abs{D U}^{n + 1} \dif r\le \int_{\Bset^{\Hset^{n + 1}}_{2 \rho} (a) \setminus \bigcup_{b \in S} \Bset^{\Hset^{n + 1}}_\delta (b)} \abs{D U}^{n + 1}\ .
\]
In particular, there exists \(\rho_* \in \mathcal R\) such that 
\begin{equation}\label{equationHomogeneousExtensionGoodBoundaryEnergy}
\begin{split}
  \int_{\partial \Bset^{\Hset^{n + 1}}_{\rho_*} (a)} \abs{D U}^{n + 1} &\le \frac{1}{\abs{\mathcal R}} \int_{\Bset^{\Hset^{n + 1}}_{2 \rho} (a) \setminus \bigcup_{b \in S} \Bset^{\Hset^{n + 1}}_\delta (b)} \abs{D U}^{n + 1} \\
  &\le \frac{4}{\rho} \int_{\Bset^{\Hset^{n + 1}}_{2 \rho} (a) \setminus \bigcup_{b \in S} \Bset^{\Hset^{n + 1}}_\delta (b)} \abs{D U}^{n + 1}\ .
\end{split}
\end{equation}
We then define the map \(V : \Hset^{n + 1}_{2 \rho}(a) \to \Rset^\nu\)
for each \(x \in \Hset^{n + 1}_{2 \rho}(a) \setminus \{a\}\) by 
\[
 V (x) 
 =
 \begin{cases}
  U (x) & \text{if \(x \in \Hset^{n + 1} \setminus \Bset^{\Hset^{n + 1}}_{\rho_*}(a)\)}\ ,\\
  U (\pi_{a, \rho} (x)) & \text{if \(x \in \Bset^{\Hset^{n + 1}}_{\rho_*}(a) \setminus \{a\}\)}\ ,
 \end{cases}
\]
where \(\pi_{a, \rho}\) is the nearest point projection from \(\Bset^{\Hset^{n + 1}}_{\rho_*}(a) \setminus \{a\}\) to \(\partial \Bset^{\Hset^{n + 1}}_{\rho_*}(a)\) as in the statement of  lemma~\ref{lemmaHomogeneousExtension}. 
By lemma~\ref{lemmaHomogeneousExtension} \eqref{lemmaHomogeneousWeaklyDifferentiable}, the map \(V\) is weakly differentiable. Since \(\rho_* \le 2 \rho\), the conclusion \eqref{equationHomogeneousExtension2} holds immediately. If we define
\[
  R = \bigl\{ b \in S \st \Bset^{\Hset^{n + 1}}_\delta (b) \subset \Bset^{\Hset^{n + 1}}_{\rho_*}  (a) \bigr\},
\]
then \eqref{eqPreservationSingularities} also holds directly. If \(b \in S \setminus R\), then, since \(\rho_* \in \mathcal R\), we have 
\[
 \Bset^{\Hset^{n + 1}}_\delta (b) \cap 
 \Bset^{\Hset^{n + 1}}_{\rho_*} (a)=\emptyset,
\]
so that \eqref{equationHomogeneousExtension1} follows again directly from the definition of \(V\). By definition of the set \(\mathcal R\), \(U (\partial \Bset^{\Hset^{n + 1}}_{\rho_*}(a)) \subset N\), and thus \(V (\Bset^{\Hset^{n + 1}}_{\rho_*}(a)) \subset N\). 
Since \(U = V\) in \(\Hset^{n + 1} \setminus \Bset^{\Hset^{n + 1}}_{\rho_*}(a)\), we have 
\[
  V^{-1} (N) = U^{-1} (N) \cup \Bset^{\Hset^{n + 1}}_{\rho_*} (a)
  \supseteq U^{-1} (N) \cup \Bset^{\Hset^{n + 1}}_{\rho} (a)\ ,
\]
and \eqref{equationHomogeneousExtensionCloseToN} is proved. From lemma~\ref{lemmaHomogeneousExtension} \eqref{lemmaHomogeneousWeakEstimate} applied with the radius \(\rho_*\)  and by  the bound \eqref{equationHomogeneousExtensionGoodBoundaryEnergy}, we obtain then for every \(\lambda > 0\)
\[
\begin{split}
  \lambda^{n + 1}\mu_{\Hset^{n + 1}} \bigl(\bigl\{ x \in & \Bset^{\Hset_{n + 1}}_{\delta} (a)  \st \abs{D V (x)}_{\Hset^{n + 1}} > \lambda \bigr\}\bigr)\\
  &\le \lambda^{n + 1}\mu_{\Hset^{n + 1}} \bigl(\bigl\{ x \in \Bset^{\Hset_{n + 1}}_{\rho_*} (a) \st \abs{D (U \circ \pi_{a, \rho}) (x)}_{\Hset^{n + 1}} > \lambda \bigr\}\bigr)\\
    &\le\frac{\sinh(\rho_*)}{n+1}\int_{\partial \Bset_{\rho_*}^{\Hset^{n + 1}}(a)} \abs{D U}^{n + 1}_{\Hset^{n + 1}} \dif \mu_{\Hset^{n + 1}}\\
  &\le \frac{4\sinh(2 \rho)}{(n+1)\rho}  \int_{\Bset^{\Hset^{n + 1}}_{2 \rho} (a) \setminus \bigcup_{b \in S} \Bset^{\Hset^{n + 1}}_\delta (b)} \abs{D U}^{n + 1}_{\Hset^{n + 1}} \dif \mu_{\Hset^{n + 1}},
\end{split}
\]
which gives the estimate \eqref{equationMarcinkiewiczBoundHomogeneousExtension}. For the remaining conclusion \eqref{equationLebesgueBoundHomogeneousExtension}, we observe that 
\[
 \Bset_{2\rho}^{\Hset^{n + 1}}(a)\setminus \Bset_{\rho_*}^{\Hset^{n + 1}}(a)
 \setminus \bigcup_{b \in S \setminus R}\Bset_\delta^{\Hset^{n + 1}}(b)
 = \Bset_{2\rho}^{\Hset^{n + 1}}(a)\setminus \Bset_{\rho_*}^{\Hset^{n + 1}}(a)
 \setminus \bigcup_{b \in S}\Bset_\delta^{\Hset^{n + 1}}(b)\ ,
\]
and thus 
\begin{multline}
\label{equationLebesgueBoundHomogeneousExtensionOuter}
 \int_{\Bset_{2\rho}^{\Hset^{n + 1}}(a)\setminus \Bset_{\rho_*}^{\Hset^{n + 1}}(a)
 \setminus \bigcup_{b \in S \setminus R}\Bset_\delta^{\Hset^{n + 1}}(b)} \abs{D V}_{\Hset^{n + 1}}^{n + 1}\dif \mu_{\Hset^{n + 1}} \\
 =\int_{\Bset_{2\rho}^{\Hset^{n + 1}}(a)\setminus \Bset_{\rho_*}^{\Hset^{n + 1}}(a)
 \setminus \bigcup_{b \in S}\Bset_\delta^{\Hset^{n + 1}}(b)} \abs{D U}_{\Hset^{n + 1}}^{n + 1} \dif \mu_{\Hset^{n + 1}} \ .
\end{multline}
In \(\Bset_{\rho_*}^{\Hset^{n + 1}}(a) \setminus \Bset_{\delta}^{\Hset^{n + 1}}(a)\), we apply lemma~\ref{lemmaHomogeneousExtension} \eqref{lemmaHomogeneousNonSingular} 
with the larger radius \(\rho_*\) and smaller radius \(\delta\) and then \eqref{equationHomogeneousExtensionGoodBoundaryEnergy} to obtain,
\begin{equation}\label{equationLebesgueBoundHomogeneousExtensionInner}
\begin{split}
 \int_{\Bset_{\rho_*}^{\Hset^{n + 1}}(a) \setminus \Bset_{\delta}^{\Hset^{n + 1}}(a)}
 &\abs{D V}_{\Hset^{n + 1}}^{n + 1} \dif \mu_{\Hset^{n + 1}}
  = \int_{\Bset_{\rho_*}^{\Hset^{n + 1}}(a) \setminus \Bset_{\delta}^{\Hset^{n + 1}}(a)}
 \abs{D V}_{\Hset^{n + 1}}^{n + 1} \dif \mu_{\Hset^{n + 1}}\\
 & = \sinh(\rho_*)\ln\frac{\tanh\left(\frac{\rho_*}{2}\right)}{\tanh\left(\frac{\delta}{2}\right)}\int_{\partial \Bset_{\rho_*}^{\Hset^{n + 1}}(a)} \abs{D U}_{\Hset^{n + 1}}^{n + 1} \dif \mu_{\Hset^{n + 1}}\\
 & \le \frac{4 \sinh(2 \rho)}{\rho} \ln\frac{\tanh (\rho) }{\tanh\left(\frac{\delta}{2}\right)}\int_{\Bset^{\Hset^{n + 1}}_{2 \rho} (a) \setminus \bigcup_{b \in S} \Bset^{\Hset^{n + 1}}_\delta (b)} \abs{D U}_{\Hset^{n + 1}}^{n + 1}\dif \mu_{\Hset^{n + 1}}\ .
\end{split}
\end{equation}
The conclusion \eqref{equationLebesgueBoundHomogeneousExtension} follows then from \eqref{equationLebesgueBoundHomogeneousExtensionOuter}
and \eqref{equationLebesgueBoundHomogeneousExtensionInner}.
\end{proof}

\section{Construction of the singular extension}
\label{sectionMainProof}

The last tool that we shall use is a covering lemma for the hyperbolic space.
\begin{lemma}[Covering of the hyperbolic space]
\label{lemmaCovering}
There exists a collection of points \(A \subset \Hset^n\) such that 
\[
  \Hset^{n + 1} = \bigcup_{a \in A} \Bset^{\Hset^{n + 1}}_\rho (a)\quad\text{and the balls}\quad\Bset^{\Hset^{n + 1}}_{\rho/2} (a),\ a\in \Hset^{n+1} \text{ are disjoint}\ .
\]
Moreover, for every \(a \in A\) and \(\sigma \in (\rho, \infty)\)
\[
  \# \bigl\{ b \in A \st b \in \Bset^{\Hset^{n + 1}}_{\sigma} (a)\bigr\} \le \frac{\int_0^{\sigma + \frac{\rho}{2}} (\sinh r)^n \dif r }{\int_0^{\frac{\rho}{2}} (\sinh r)^n \dif r}.
\]
\end{lemma}

This lemma is a weak form of the Besicovitch covering theorem for the hyperbolic space; 
however, due to the hyperbolic geometry, the number of balls covering a given point 
not only depends on the dimension of the space but also exponentially on the radius \(\rho\).

\begin{proof}%
[Proof of lemma~\ref{lemmaCovering}]%
Let \(\{\Bset^{\Hset^{n + 1}}_{\rho/2} (a)\}_{a \in A}\) be a maximal family of disjoint hyperbolic open balls in \(\Hset^{n + 1}\). 
Then if \(x \in \Hset^{n + 1} \setminus A\), we have 
\[
  \Bset^{\Hset^{n + 1}}_{\rho/2} (x) \cap  \Bset^{\Hset^{n + 1}}_{\rho/2} (a) \ne \emptyset\ ,
\]
and hence 
\[
 x \in \Bset^{\Hset^{n + 1}}_{\rho} (a)\ .
\]
Therefore, we have
\[
  \Hset^{n + 1}\subset \bigcup_{a \in A} \Bset^{\Hset^{n + 1}}_\rho (a)\ .
\]
Moreover if \(b \in A\) and \(b \in \Bset^{\Hset^{n + 1}}_{\sigma} (a)\), then 
\[
 \Bset^{\Hset^{n + 1}}_{\rho/2} (b) \subset \Bset^{\Hset^{n + 1}}_{\sigma + \rho/2} (a)\ .
\]
We have thus
\[
  \bigcup_{b \in A \cap \Bset^{\Hset^{n + 1}}_{\sigma} (a)} \Bset^{\Hset^{n + 1}}_{\rho/2} (b)
  \subset \Bset^{\Hset^{n + 1}}_{\sigma + \rho /2} (a)\ ,
\]
so that, since the balls in the collection are disjoint,  
\[
 \# \bigl \{ b \in A \st b \in \Bset^{\Hset^{n + 1}}_{\sigma} (a)\bigr\}\ 
 \mu_{\Hset^{n + 1}} \bigl(\Bset^{\Hset^{n + 1}}_{\rho/2} (a)\bigr)
 \le \mu_{\Hset^{n + 1}} \bigl(\Bset^{\Hset^{n + 1}}_{\sigma + \rho/2} (a)\bigr)\ ,
\]
from which the desired estimate follows since for each \(\tau \in (0, \infty)\),
\[
 \mu_{\Hset^{n + 1}} \bigl(\Bset^{\Hset^{n + 1}}_{\tau} (b)\bigr)
 =\abs{\Sset^{n}} \int_0^{\tanh \frac{\tau}{2}} \frac{2^{n + 1} s^{n}}{(1 - s^2)^{n + 1}}\dif s
 =\abs{\Sset^{n}} \int_0^\tau (\sinh r)^n \dif r\ .\qedhere
\]
\end{proof}

We have now all the tools to prove the main result of the present work.

\begin{proof}%
[Proof of theorem~\ref{theoremMain}]
Let \(u \in W^{n/(n + 1), n + 1} (\Sset^n; M)\).

\step{Step 1. Hyperharmonic extension.}
Let \(\mathfrak{H} u\) be given by the hyperharmonic extension (see \eqref{equationConformalExtension}).
By proposition~\ref{propositionConformalExtension}, \(\mathfrak{H}u \in W^{1, n + 1} (\Bset^{n + 1}; \Rset^\nu)\) admits \(u\) as a trace on \(\partial \Bset^{n + 1}\).

\step{Step 2. Choosing a scale with many good radii.} 
Since the manifold \(M\) is embedded as a compact subset of \(\Rset^\nu\), there exists \(\iota > 0\) and a Lipschitz retraction \(\pi_M : N \to M\), where
\[
 N = \bigl\{ y \in \Rset^\nu \st \dist (y, M) < \iota\bigr\}\ .
\]
By the classical Chebyshev inequality and the good radii estimate (proposition~\ref{propositionGoodRadiusFractional}), we have 
\[
\begin{split}
 \bigabs{\bigl\{ r \in (0, 2 \rho) \st \mathfrak{H}u (\partial \Bset_r^{\Hset^{n + 1}}(a)) \not \subseteq N\bigr\}}
 \le \frac{1}{\iota} \int_0^{2\rho}  \norm{\dist (\mathfrak{H} u, u (\Sset^n))}_{L^\infty (\partial\Bset^{\Hset^{n + 1}}_r(a))}\dif r\\
 \le \frac{C_1 2\rho}{\iota(1 + (2\rho)^\frac{1}{n + 1})} \left(\int_{\Sset^n} \int_{\Sset^n} \frac{\abs{u (y) - u (z)}^{n + 1}}{\abs{y - z}^{2 n}}\dif y \dif z \right)^\frac{1}{n + 1} .
\end{split}
\]
We take 
\begin{equation}
\label{eqDefinitionRho}
  \rho = \frac{1}{2} \Bigl(\frac{8 C_1}{\iota}\Bigr)^{n+1} \int_{\Sset^n} \int_{\Sset^n} \frac{\abs{u (y) - u (z)}^{n + 1}}{\abs{y - z}^{2 n}}\dif y \dif z\ .
\end{equation}
so that, for every \(a \in \Hset^n\),
\begin{equation}
\label{condonrho}
\bigabs{\bigl\{ r \in (\rho, 2 \rho) \st \mathfrak{H}u (\partial \Bset_r^{\Hset^{n + 1}}(a)) \not \subset N\bigr\}} \le \frac{\rho}{4}\ .
\end{equation}

\step{Step 3. Fixing a good covering.} 
By lemma~\ref{lemmaGoodRegion} \eqref{itemAsymptoticBehaviour}, there exists a hyperbolic ball \(\Bset_{\Bar{\rho}}^{\Hset^{n + 1}} (0)\) such that for every \(x \in \Hset^{n + 1} \setminus \Bset_{\Bar{\rho}}^{\Hset^{n + 1}} (0)\),
\[
  \dist \bigl(\mathfrak{H}u (x), u (\Sset^n)\bigr) < \iota\ .
\]
Let \(A \subset \Hset^{n + 1}\) be the collection of centers of balls given lemma~\ref{lemmaCovering} and let \(\Tilde{S} = A \cap \Bset^{\Hset^{n + 1}}_{\Bar{\rho} + \rho} (0)\).

In view of the conclusion of lemma~\ref{lemmaCovering}, we can partition \(\Tilde S\) like in the proof of the classical Besicovitch covering lemma and write 
\[
 \Tilde{S} = \bigcup_{j = 1}^Q \Tilde{S}_j\ ,
\]
with
\begin{equation}
\label{ineqQ}
  Q \le \frac{\int_0^{9 \rho/2} (\sinh r)^n \dif r }{\int_0^{\rho/2} (\sinh r)^n \dif r}\ ,
\end{equation}
so that if \(j \in \{1, \dotsc, Q\}\), \(a, b \in \Tilde{S}_j\) and \(a \ne b\), 
then \(\Bset^{\Hset^{n + 1}}_{2\rho} (a) \cap \Bset^{\Hset^{n + 1}}_{2\rho} (b) = \emptyset\), 
or equivalently, \(b \not \in \Bset^{\Hset^{n + 1}}_{4\rho} (a)\).

\step{Step 4. Improving the function iteratively on balls.}
In view of proposition~\ref{propositionOneBall}, we fix
\begin{equation}
\label{ineqdelta}
  \delta =  \min\biggl(\frac{\rho\int_0^{\rho/2} (\sinh r)^n \dif r}{4\int_0^{5 \rho/2} (\sinh r)^n \dif r }, \frac{\rho}{2}\biggr)\ .
\end{equation}

By iteratively applying proposition~\ref{propositionOneBall}, we shall construct inductively maps \(U_0, \dotsc, U_Q \in W^{1, 1}_{\mathrm{loc}} (\Hset^{n + 1}; \Rset^\nu)\) and sets \(S_0, \dotsc, S_Q \subseteq \Tilde{S}\), such that
we have for \(q=1,\ldots,Q\)
\begin{enumerate}[(a)]
\vspace{1ex}
 \item \label{iterateGoodRegion} \(U_q^{-1}(N) \supseteq (\mathfrak{H}u)^{-1}(N) \cup \bigcup_{j = 1}^q \bigcup_{a \in \Tilde{S}_q} \Bset^{\Hset^{n + 1}}_{\rho} (a)\ \),
 \vspace{1ex}
 \item \label{iterateSingularEstimate} \(\displaystyle \sum_{a \in S_q} \norm{D U_q}_{L^{n + 1, \infty} (\Bset^{\Hset^{n + 1}}_{\delta} (a))}^{n + 1}
 \le \eta \frac{\kappa^{q} - 1}{\kappa - 1} \int_{\Hset^{n + 1}} \abs{D\mathfrak{H} u}^{n + 1} \dif \mu_{\Hset^{n + 1}}\ , \)\\
 \vspace{1ex}
 where
 \[
  \norm{D U_q}_{L^{n + 1, \infty} (\Bset^{\Hset^{n + 1}}_{\delta} (a))}^{n + 1}
  =\sup_{\lambda > 0} \lambda^{n + 1} \mu_{\Hset^{n + 1}} \bigl(\{x \in \Bset^{\Hset^{n + 1}}_{\delta} (a)) \st  \abs{D U_q (x)}_{\Hset^{n + 1}} > \lambda\} \bigr)\ ,
 \]
 \item \label{iterateRegularEstimate} \(\displaystyle
 \int_{\Hset^{n + 1} \setminus \bigcup_{a \in S_q} \Bset^{\Hset^{n + 1}}_{\delta} (a)} \abs{D U_q}_{\Hset^{n + 1}}^{n + 1} \le \kappa^q \int_{\Hset^{n + 1}} \abs{D\mathfrak{H} u}^{n + 1} \dif \mu_{\Hset^{n + 1}}\ ,\)
 \vspace{1ex}
 \item \(U_q = u\) on \(\partial \Bset^{n + 1}\) in the sense of traces,
\end{enumerate}
where 
\begin{align}
\label{equationkappagamma}
  \kappa &= 1 + \frac{4\sinh (2 \rho)}{\rho} \ln \frac{\tanh (\rho)}{\tanh \frac{\delta}{2}}&
  &\text{ and }&
  \eta = \frac{4\sinh (2\rho)}{\rho}\ .
\end{align}

We first set \(U_0 = \mathfrak{H} (u)\) and \(S_0 = \emptyset\) and we observe immediately that \(U_0\) satisfies all the properties.

We now assume that \(q \in \{0, \dotsc, Q - 1\}\) and that \(U_0, \dotsc, U_{q}\) and \(S_0, \dotsc, S_{q}\) have already been constructed.
By \eqref{iterateGoodRegion} and by \eqref{condonrho}, we have, for each \(a \in \Tilde{S}_{q + 1}\),
\begin{multline*}
\bigabs{\bigl\{ r \in (\rho, 2 \rho) \st U_q (\partial \Bset_r^{\Hset^{n + 1}}(a)) \not \subseteq N\bigr\}}\\
\le 
\bigabs{\bigl\{ r \in (\rho, 2 \rho) \st \mathfrak{H}u (\partial \Bset_r^{\Hset^{n + 1}}(a)) \not \subseteq N\bigr\}} \le \frac{\rho}{4}\ . 
\end{multline*}
Moreover, since \(S_q \subseteq \Tilde{S}\), we have for each \(a \in \Tilde{S}_{q + 1}\), in view of the construction of \(A\) by lemma~\ref{lemmaCovering} and the definition of \(\delta\),
\[
\begin{split}
 \# (S_{q} \cap \Bset^{\Hset^{n + 1}}_{2 \rho} (a))
 &\le \# (\Tilde{S} \cap \Bset^{\Hset^{n + 1}}_{2 \rho} (a))
 \le \# (A \cap \Bset^{\Hset^{n + 1}}_{2 \rho} (a))\\
 &\le \frac{\int_0^{5 \rho/2} (\sinh r)^n \dif r }{\int_0^{\rho/2} (\sinh r)^n \dif r}
 \le \frac{ \rho}{4\delta}\ .
\end{split}
\]
For each \(a \in \Tilde{S}_q\), we apply proposition~\ref{propositionOneBall} to the map \(U_q\) with the singularity set \(S_{q}\), and we denote the resulting map \(U_{q + 1}^a\) and the associated set \(R_{q + 1}^a\).
We define now \(U_{q + 1} : \Hset^{n + 1} \to \Rset^\nu\) for \(x \in \Hset^{n + 1}\) by 
\[
 U_{q + 1} (x)
 =
 \begin{cases}
    U_{q + 1}^a (x) & \text{if \(x \in \Bset^{\Hset^{n + 1}}_{2 \rho} (a)\) for \(a \in \Tilde{S}_{q + 1}\)}\ ,\\
    U_{q} (x) & \text{otherwise}\ ,
 \end{cases}
\]
and 
\[
 S_{q + 1} = S_{q} \cup \Tilde{S}_{q + 1} \setminus \bigcup_{a \in \Tilde{S}_{q + 1}}R_{q + 1}^a.
\]

By the disjointness of the family \(\{\Bset^{\Hset^{n + 1}}_{2 \rho} (a)\}_{a \in S_q}\), the map \(U_q\) is well-defined. Moreover, by proposition~\ref{propositionOneBall} \eqref{equationHomogeneousExtensionCloseToN},  \(U_q \in W^{1, 1}_{\mathrm{loc}} (\Hset^{n + 1}; \Rset^\nu)\).

By proposition~\ref{propositionOneBall} \eqref{equationHomogeneousExtensionCloseToN} and by our induction assumption \eqref{iterateGoodRegion}, we have 
\[
  U_{q + 1}^{-1} (N) \supseteq U_{q}^{-1} (N) \cup \bigcup_{a \in \Tilde{S}_{q + 1}}
  \Bset^{\Hset^{n + 1}}_{\rho} (a)
  \supseteq
  (\mathfrak{H}u)^{-1}(N) \cup \bigcup_{j = 1}^{q + 1} \bigcup_{a \in \Tilde{S}_q} \Bset^{\Hset^{n + 1}}_{\rho} (a)
\]
so that \eqref{iterateGoodRegion} holds for \(q + 1\).

For each \(a \in \Tilde{S}_{q + 1}\),  we have by proposition~\ref{propositionOneBall} \eqref{equationMarcinkiewiczBoundHomogeneousExtension},
\[
\begin{split}
  \norm{D U_{q + 1}}_{L^{n + 1, \infty}(\Bset^{\Hset^{n + 1}}_\delta(a))}^{n+1} 
  \le \eta \int_{\Bset^{\Hset^{n + 1}}_{2 \rho}(a) \setminus \bigcup_{b \in S_q}\Bset_\delta^{\Hset^{n + 1}}(b)} \abs{D U_q}_{\Hset^{n + 1}}^{n + 1} \dif \mu_{\Hset^{n + 1}}\ ,
\end{split}
\]
and therefore, by the induction assumption \eqref{iterateRegularEstimate}
\begin{equation}
\label{ineqIterateSingularEstimateNew}
\begin{split}
 \sum_{a \in S_{q + 1}}
 \norm{D U_{q + 1}}_{L^{n + 1, \infty}(\Bset^{\Hset^{n + 1}}_\delta(a))}^{n+1}
 &\le \sum_{a \in S_{q + 1}} \eta \int_{\Bset^{\Hset^{n + 1}}_{2 \rho}(a) \setminus \bigcup_{b \in S_q}\Bset_\delta^{\Hset^{n + 1}}(b)}\hspace{-2em} \abs{D U_q}^{n + 1}_{\Hset^{n + 1}}\dif \mu_{\Hset^{n + 1}}\\
 &\le \eta \kappa^{q - 1} \int_{\Hset^{n + 1}} \abs{D \mathfrak{H} u}_{\Hset^{n + 1}}^{n + 1}\dif \mu_{\Hset^{n + 1}}\ .
\end{split}
\end{equation}
On the other hand, by  proposition~\ref{propositionOneBall} \eqref{equationHomogeneousExtension1} for every \(b \in S_{q + 1} \setminus \Tilde{S}_{q + 1}\), we have 
\begin{equation*}
   U_{q + 1} \vert_{ \Bset^{\Hset^{n + 1}}_\delta (b)} = U_q \vert_{ \Bset^{\Hset^{n + 1}}_\delta (b)}.
\end{equation*}
It follows thus by the induction assumption \eqref{iterateSingularEstimate} that 
\begin{multline}
\label{ineqIterateSingularEstimateOld}
 \sum_{a \in S_{q + 1}\setminus \Tilde{S}_{q + 1}}
 \norm{D U_{q + 1}}_{L^{n + 1, \infty}(\Bset^{\Hset^{n + 1}}_\delta (a))}^{n+1}\\
 \le \eta \frac{\kappa^{q - 1} - 1}{\kappa - 1} \int_{\Bset^{\Hset^{n + 1}}_{2 \rho}(a) \setminus \bigcup_{b \in S_q}\Bset_\delta^{\Hset^{n + 1}}(b)} \abs{D U_q}_{\Hset^{n + 1}}^{n + 1}\dif \mu_{\Hset^{n + 1}}\ , 
\end{multline}
The assertion \eqref{iterateSingularEstimate} for \(q + 1\) follows now from \eqref{ineqIterateSingularEstimateNew} and \eqref{ineqIterateSingularEstimateOld}.

Similarly, for each \(a \in \Tilde{S}_{q + 1}\), by proposition~\ref{propositionOneBall} \eqref{equationLebesgueBoundHomogeneousExtension}, we have 
\begin{multline}
\label{ineqIterateRegularEstimateNew}
  \int_{\Bset^{\Hset^{n + 1}}_{2 \rho}(a) \setminus \bigcup_{b \in S_{q + 1}} \Bset^{\Hset^{n + 1}}_\delta (b)} \abs{D U_{q + 1}}_{\Hset^{n + 1}}^{n + 1}\dif \mu_{\Hset^{n + 1}}
 \hspace*{\stretch{1}}\\
  \le \kappa
  \int_{\Bset^{\Hset^{n + 1}}_{2 \rho}(a) \setminus \bigcup_{b \in S_{q}} \Bset^{\Hset^{n + 1}}_\delta (b)} \abs{D U_q}_{\Hset^{n + 1}}^{n + 1}\dif \mu_{\Hset^{n + 1}}\ ,
\end{multline}
moreover by construction \(U_{q + 1} = U_q\) on \(\Hset^{n + 1} \setminus \bigcup_{a \in \Tilde{S}_q}\Bset^{\Hset^{n + 1}}_{2\rho} (a)\) and by proposition~\ref{propositionOneBall} \eqref{eqPreservationSingularities}
\[
 \Hset^{n + 1} \setminus \bigcup_{a \in \Tilde{S}_{q + 1}} \Bset^{\Hset^{n + 1}}_{2\rho} (a) \setminus \bigcup_{b \in S_{q + 1}} \Bset^{\Hset^{n + 1}}_{\delta} (b)
 =\Hset^{n + 1} \setminus \bigcup_{a \in \Tilde{S}_{q + 1}} \Bset^{\Hset^{n + 1}}_{2\rho} (a) \setminus \bigcup_{b \in S_{q}} \Bset^{\Hset^{n + 1}}_{\delta} (b)\ ,
\]
so that 
\begin{multline}
\label{ineqIterateRegularEstimateOld}
  \int_{ \Hset^{n + 1} \setminus \bigcup_{a \in \Tilde{S}_{q + 1}} \Bset^{\Hset^{n + 1}}_{2\rho} (a) \setminus \bigcup_{b \in S_{q + 1}} \Bset^{\Hset^{n + 1}}_{\delta} (b)} \abs{D U_{q + 1}}_{\Hset^{n + 1}}^{n + 1}\dif \mu_{\Hset^{n + 1}}
 \hspace*{\stretch{1}}\\
  =
  \int_{\Hset^{n + 1} \setminus \bigcup_{a \in \Tilde{S}_{q + 1}} \Bset^{\Hset^{n + 1}}_{2\rho} (a) \setminus \bigcup_{b \in S_{q}} \Bset^{\Hset^{n + 1}}_{\delta} (b)} \abs{D U_q}_{\Hset^{n + 1}}^{n + 1}\dif \mu_{\Hset^{n + 1}}\ .
\end{multline}
The estimates \eqref{ineqIterateRegularEstimateNew} and \eqref{ineqIterateSingularEstimateOld} together with our induction assumption \eqref{iterateRegularEstimate} imply that \eqref{iterateRegularEstimate} holds for \(q + 1\).

Finally, since the set \(\Tilde{S}_q\) is finite, 
the maps \(U_{q + 1}\) and \(U_q\) coincide outside a compact subset of \(\Hset^{n + 1}\), 
and thus share the same trace on \(\Sset^{n}\).

\step{Step 5. Projection and Euclidean estimates.}
By assumption, 
\[
 (\mathfrak{H}u)^{-1} (N) \supseteq \Hset^{n + 1} \setminus \Bset^{\Hset^{n + 1}}_{\Bar{\rho}} (0) \supseteq \Hset^{n + 1} \setminus \bigcup_{b \in \Tilde{S}} \Bset^{\Hset^{n + 1}}_{\rho} (b)\ .
\]
Therefore by \eqref{iterateGoodRegion}, we have \(U_{Q} \in N\) almost everywhere in \(\Hset^{n + 1}\).
Therefore, we can set \(U = \pi_M \circ U_Q\). 
Since \(\pi_M\) is a retraction on \(N\), we have \(U \in M\) almost everywhere in \(\Hset^{n + 1}\).

Since the retraction \(\pi_M\) is Lipschitz, \(U \in W^{1, 1}_\mathrm{loc} (\Hset^{n + 1}; \Rset^\nu)\).
We have for each \(a \in S_Q\) and \(x \in \Bset^{\Hset^{n + 1}}_{\rho} (a)\),
\[
 \frac{1 - \abs{x}^2}{1 - \abs{a}^2}
 =\frac{\bigl(\cosh \frac{d_{\Hset^{n + 1}} (a, 0)}{2}\bigr)^2}{\bigl(\cosh \frac{d_{\Hset^{n + 1}} (x, 0)}{2}\bigr)^2}\ .
\]
Since \(\abs{d_{\Hset^{n + 1}} (a, 0) - d_{\Hset^{n + 1}} (x, 0)} \le \delta\), we have thus 
\begin{equation}
\label{eqHyperbolicCloseIsometry}
  e^{-\delta} \le \frac{1 - \abs{x}^2}{1 - \abs{a}^2} \le e^{\delta}\ .
\end{equation}
Therefore, we have for \(\lambda > 0\), in view of \eqref{eqHyperbolicMeasure}, \eqref{eqHyperbolicDualNorm} and \eqref{eqHyperbolicCloseIsometry},
\begin{equation}
\label{ineqWeakEuclideanSingular}
\begin{split}
 \lambda^{n + 1} &\bigabs{\bigl\{ x \in \Bset^{\Hset^{n + 1}}_{\delta} (a) \st \abs{D U (x)} > \lambda \bigr\}}\\
 &\le \Bigl(\frac{\lambda 2 e^{\delta}}{1 - \abs{a}^2}\Bigr)^{n + 1} \mu_{\Hset^{n + 1}} \bigl(\bigl\{ x \in \Bset^{\Hset^{n + 1}}_{\delta} (a) \st 2e^{\delta}\abs{D U (x)}_{\Hset^{n + 1}} > (1 - \abs{a}^2)\lambda \bigr\}\bigr)\\
 &\le e^{2(n + 1)\delta} \norm{D U}_{L^{n + 1, \infty} (\Bset^{\Hset^{n + 1}}_{\delta} (a))}^{n + 1}\ .
\end{split}
\end{equation}
On the other hand, by the classical Chebyshev inequality, we have by \eqref{eqHyperbolicMeasure} and \eqref{eqHyperbolicDualNorm},
\begin{equation}
\label{ineqWeakEuclideanRegular}
\begin{split}
\lambda^{n + 1}
\Bigabs{\Bigl\{\Hset^{n + 1} \setminus \bigcup_{a \in S_q} \Bset^{\Hset^{n + 1}}_{\delta} (a)
\st \abs{D U(x)} > \lambda \Bigr\}}\hspace{-8em}&\\
  &\le\int_{\Hset^{n + 1} \setminus \bigcup_{a \in S_q} \Bset^{\Hset^{n + 1}}_{\delta} (a)} \abs{D U}^{n + 1} \\
  &= \int_{\Hset^{n + 1} \setminus \bigcup_{a \in S_q} \Bset^{\Hset^{n + 1}}_{\delta} (a)} \abs{D U}_{\Hset^{n + 1}}^{n + 1}\dif \mu_{\Hset^{n + 1}}\ .
\end{split}
\end{equation}

By \eqref{ineqWeakEuclideanSingular}, \eqref{ineqWeakEuclideanRegular}, \eqref{iterateRegularEstimate} and \eqref{iterateSingularEstimate}, we also deduce that 
\begin{equation}
\label{ineqAlmostFinalEstimate}
\begin{split}
\lambda^{n + 1} &\bigabs{\bigl\{x \in \Bset^{n + 1} \st \abs{D U (x)} > \lambda \bigr\}}\\
 &\le C_2 e^{2(n + 1)\delta} \Bigl(\kappa^Q + \eta \frac{\kappa^{Q - 1} - 1}{\kappa - 1}\Bigr)
 \Bigl(\int_{\Sset^n} \int_{\Sset^n} \frac{\abs{u (y) - u (z)}^{n + 1}}{\abs{y - z}^{2 n}} \dif y \dif z \Bigr)^\frac{n}{n + 1}\\
 &\le C_2 e^{2(n + 1)\delta} \bigl(\kappa^Q + Q \kappa^{Q - 1}\bigr)
 \Bigl(\int_{\Sset^n} \int_{\Sset^n} \frac{\abs{u (y) - u (z)}^{n + 1}}{\abs{y - z}^{2 n}} \dif y \dif z \Bigr)^\frac{n}{n + 1}\ ,
\end{split}
\end{equation}
where the constant \(C_2\) comes from the Lipschitz constant of the retraction \(\pi_M\) and 
the estimate on the hyperharmonic extension of proposition~\ref{propositionConformalExtension} \eqref{estimateConformalExtension}.
This implies that \(U \in W^{1, p} (\Bset^{n + 1} \setminus \Bset^{\Hset^{n + 1}}_{\Bar{\rho} + \rho}, M)\) for every \(p \in [1, n + 1)\) and that \(U = u\) on \(\Sset^n\) in the sense of traces.

It remains to estimate the constants in \eqref{ineqAlmostFinalEstimate}.
In view of \eqref{ineqQ}, \eqref{ineqdelta} and \eqref{equationkappagamma},
we have 
\begin{align*}
  Q &\le C_3 e^{4 \rho}\ ,&
  \delta &\ge C_4 \frac{\rho}{e^{3 \rho}}\ ,\\
  \eta &\le C_5 \frac{e^{2 \rho}}{\rho}\ ,&
  \kappa &\le C_6 \frac{e^{5\rho}}{\rho^2}\ ,
\end{align*}
from which it follows by \eqref{eqDefinitionRho} and \eqref{ineqAlmostFinalEstimate} that 
\[
 \norm{D U}^{n + 1}_{L^{n + 1, \infty}}
 \le C \exp \Bigl(C \exp \Bigl(C \int_{\Sset^n} \int_{\Sset^n} \frac{\abs{u (y) - u (z)}^{n + 1}}{\abs{y - z}^{2 n}} \dif y \dif z\Bigr) \Bigr)\ ,
\]
for some constant \(C > 0\).
\end{proof}

\begin{remark}
\label{remarkSmooth}
In general the map \(U\) which is constructed is continuous in \(\Bset^{n} \setminus \Tilde{S}\), 
where the set \(\Tilde{S}\) constructed in the proof of theorem~\ref{theoremMain} is \emph{finite}. By a standard additional regularization argument it is possible to take \(U\) to be smooth in \(\Bset^{n} \setminus \Tilde{S}\).
If moreover \(u \in C^{\infty} (\Sset^n; M)\), then \(U\) can be taken to be smooth on \(\overline{\Bset^{n + 1}} \setminus \Tilde{S}\).
\end{remark}

\begin{remark}
\label{remarkGromovSchoen}
The proof of theorem~\ref{theoremMain} only requires from \(M\) that there is a Lipschitz retraction from a uniform neighborhood \(N\) of the set \(M\) to \(M\).
In particular, theorem~\ref{theoremMain} holds thus for compact singular spaces studied in \cite{GromovSchoen1992}.
\end{remark}

\section{Extension with constant regularity} 
\label{sectionConstantRegularity}

As a consequence of theorem~\ref{theoremMain} and of the Sobolev embedding \(W^{1,n} (\Sset^{n + 1}; \Rset^\nu) \subset W^{\frac{n}{n + 1}, n + 1} (\Sset^{n + 1}; \Rset^\nu)\), we have the following generalization of the controlled critical extension result of \cite{PetracheRiviere2015}:

\begin{theorem}
\label{theoremConstantRegularity}
Let \(n \in \Nset_*\) and let \(M \subset \Rset^{\nu}\) be a compact embedded manifold.
There exists a function \(\gamma : [0, \infty) \to [0, \infty)\) such that for every \(u \in W^{1, n} (\Sset^n; M)\), there exists \(U \in W^{1, (n + 1, \infty)} (\Bset^{n + 1}; M)\) such that
we have \(u = U\) in the sense of traces on \(\mathbb S^n\) and for every \(\lambda > 0\),
\[
\lambda^{n + 1} \bigabs{\bigl\{ x \in \Bset^{n + 1} \st \abs{D U (x)} > \lambda\bigr\}} 
  \le \gamma \Bigl(\int_{\Sset^n}\abs{D u}^n \Bigr)\ .
\]
\end{theorem}

Theorem~\ref{theoremConstantRegularity} can also be proved directly, following the lines of the proof of theorem~\ref{theoremMain} in \S \ref{sectionMainProof} above.
The analytic properties of the hyperharmonic extension of proposition~\ref{propositionConformalExtension} can be proved directly with a \(W^{1, n}\) bound.
We give here a direct proof of the counterpart of proposition~\ref{propositionGoodRadiusFractional} for \(W^{1, n}(\Sset^{n};\Rset^\nu)\), which has a somewhat shorter proof.

\begin{proposition}
\label{lemmaGoodRadius}
There exists \(C > 0\) depending only on the dimension such that for every \(\rho > 0\) 
and for every \(a\in\Hset\) we have
\[
 \int_0^\rho  \norm{\dist (\mathfrak{H}u, u (\Sset^n))}_{L^\infty (\partial\Bset^{\Hset^{n + 1}}_r(a))}\dif r
 \le C \left(\int_{\Sset^{n}} \abs{D u}^n \right)^\frac{1}{n} \rho^{1 - \frac{1}{n}}\ .
\]
\end{proposition}
\begin{proof}
As in the proof of proposition~\ref{propositionGoodRadiusFractional}, we assume without loss of generality that \(a = 0\) and we set \(R = \tanh \frac{\rho}{2}\).

By \eqref{ineqPointwiseDistanceOscillationComposite} and by the Poincar\'e inequality on \(W^{1, 1} (\Sset^{n}; \Rset^\nu)\), this implies that for each \(x \in \Bset^{n + 1}\),
\[
 \frac{\dist \bigl(\mathfrak{H} u (x), u (\Sset^{n})\bigr)}{1 - \abs{x}^2} \le  \frac{C_1}{(1 - \abs{x}^2)} \int_{\Sset^n} \abs{D (u \circ T_{x}^{-1})}\ .
\]
By the chain rule and a change of variable in the integral, in view of \eqref{eqMobiusDerivative}, we obtain
\[
\begin{split}
  \int_{\Sset^n} \abs{D (u \circ T_{x}^{-1}) (y)}\dif y
  &=\int_{\Sset^n} \bigl(\abs{Du} \circ T_{x}^{-1}\bigr) \abs{D T_{x}^{-1}}\\
  &=\int_{\Sset^n} \abs{Du}\, \abs{D T_{x}}^{n - 1}
  = \int_{\Sset^n} \frac{(1 - \abs{x}^2)^{n - 1} 
  \abs{D u (y)}}{\abs{y - x}^{2 (n - 1)}}\dif y\ .
\end{split}
\]
Let \(\abs{D u}^*: \Sset^n \to [0, \infty]\) denotes the spherical rearrangement of the function \(\abs{D u} : \Sset^n \to [0, \infty]\) with respect to a fixed point \(e \in \Sset^n\), that is, for each \(\lambda > 0\), 
\(\abs{D u}^* (y)^{-1} ((\lambda, \infty))\) is a geodesic ball of \(\Sset^n\) centered at the point \(e\) that has the same measure as \(\abs{D u} (y)^{-1} ((\lambda, \infty))\). 
By the classical Hardy--Littlewood rearrangement inequality, we have if \(r = \abs{x}\)
\[
 \frac{\dist (\mathfrak{H} u (x), u (\Sset^{n}))}{1 - r^2} 
 \le C_1 \int_{\Sset^n} \frac{(1 - r^2)^{n - 2} \abs{D u}^* (y)}{\abs{y - re}^{2 (n - 1)}}\dif y\ .
\]
Since \(1 - r^2 \le 2 \abs{y - re}\), by \eqref{ineqTriple},
we now compute for every \(e \in \Sset^{n}\) and \(y \in \Bset^{n + 1}\),
\[
\begin{split}
 \int_0^R \frac{(1 - r^2)^{n - 2}}{\abs{y - r e}^{2 (n - 1)}} \dif r
 & \le \int_0^R \frac{2^{n - 2}}{\abs{y - r e}^{n}} \dif r\\
 & \le \int_0^R \frac{2^{n - 2} 3^n}{(\abs{y - e} + 1 - r)^n} \dif r \le \frac{2^{n - 2}3^n}{n (\abs{y - e} + 1 - R)^{n - 1}}\ .
\end{split}
\]
We have thus proved that 
\[
 \int_0^R  \frac{\norm{\dist (\mathfrak{H} u, u (\Sset^n))}_{L^\infty (\partial \Bset^{n + 1}_r)}}{1 - r^2}\dif r
 \le C_1\frac{2^{n - 2} 3^n}{n}\int_{\Sset^n} \frac{\abs{D u}^* (y)}{(\abs{y - e} + 1 - R)^{n - 1}}\dif y\ .
\]
Now, we observe that
\[
\int_{\Sset^n} \frac{1}{(\abs{y - e} + 1 - R)^{n}}\dif y
  \le C_2 \ln \frac{1}{1 - R}\ .
\]
Hence, by the classical H\"older inequality, 
\[
\begin{split}
 \int_0^R  \frac{\norm{\dist (\mathfrak{H} u, u (\Sset^n))}_{L^\infty (\partial \Bset^{n + 1}_r)}}{1 - r^2}\dif r
 &\le C_1 \left(\int_{\Sset^{n}} (\abs{D u}^*)^n \right)^\frac{1}{n} \Bigl( C_2 \ln \frac{1}{1 - R} \Bigr)^{1 - \frac{1}{n}}\\
 &=C_1 \left(\int_{\Sset^{n}} \abs{D u}^n \right)^\frac{1}{n} \Bigl( C_2 \ln \frac{1}{1 - R} \Bigr)^{1 - \frac{1}{n}}\ ,
\end{split}
\]
in view of the Cavalieri principle for the rearrangement. The conclusion follows.
\end{proof}


\section{High integrability case}

In this section we explain how a controlled extension can be obtained in \(W^{1, p}\).

\begin{theorem}\label{propositionSupercritical}
Let \(n \in \Nset_*\), \(p > n + 1\) and \(M \subset \Rset^\nu\) a compact Riemannian manifold. 
There exists \(\gamma : \Rset^+ \to \Rset^+\) such that if \(u \in W^{1 - 1/p, p} (\Sset^{n}; M)\) is homotopic to a constant map as a continuous map, 
then there exists \(U \in W^{1, p} (\Bset^{n+1}; M)\) such that \(U = u\) in the sense of traces on \(\Sset^{n}\) and
\[
\int_{\Bset^{n + 1}} \abs{D U}^p
  \le \gamma \Bigl(\int_{\Sset^n}\int_{\Sset^n} \frac{\abs{u (y) - u (z)}^p}{\abs{y - z}^{n + p - 1}} \dif y \dif z \Bigr)\ .
\] 
\end{theorem}
\begin{proof}
Let \(\theta \in (0, \infty)\).
Since \(p > n + 1\) and the manifold \(M\) is compact, by the fractional Sobolev--Morrey embedding and the classical Ascoli--Arzel\`a compactness theorem, the set 
\begin{multline*}
 \mathcal{B}_{\theta} = \Bigl\{ u \in W^{1 - \frac{1}{p}, p} (\Bset^{n + 1}; M) \st \int_{\Sset^n}\int_{\Sset^n} \frac{\abs{u (y) - u (z)}^p}{\abs{y - z}^{n + p - 1}} \dif y \dif z \le \theta\\
 \text{and } u \text{ is homotopic to a constant map}\Bigr\}\ ,
\end{multline*}
is a compact subset of \(C (\Sset^{n}; M)\) with respect to the uniform distance.

There exists \(\iota > 0\) such that a Lipschitz retraction \(\pi_M : N \to M\) is well-defined on 
\[
 N = \bigl\{ y \in \Rset^\nu \st \dist (y, M) < \iota \bigr\}\ .
\]
Since \(C^\infty (\Sset^{n}; M)\) is dense in \(C (\Cset^{n}; M)\) and since \(\mathcal{B}_{\theta}\) is compact with respect to the uniform metric, 
there exists maps \(u_1, \dotsc, u_Q \in C^\infty (\Sset^n; M)\) such that for every \(u \in \mathcal{B}_{\theta}\), 
there exists \(q \in \{1, \dotsc, Q\}\) such that \(\norm{u - u_q}_{L^\infty} < \iota\). 
Since maps in \(\mathcal{B}_{\theta}\) are homotopic to a constant, 
we can assume without loss of generality that \(u_1, \dotsc, u_Q\) are homotopic to a constant. 
(A map that would not be homotopic to a constant would not satisfy \(\norm{u - u_q}_{L^\infty} < \iota\) for any \(u \in \mathcal{B}_\theta\).)
We let \(U_1, \dotsc, U_Q \in C^\infty (\Bset^{n + 1}; M)\) be smooth extensions of \(u_1, \dotsc, u_Q\) 
such that \(U_q (x) = u_q (x/\abs{x})\) if \(\abs{x} \ge 1/2\).

We fix now \(u \in \mathcal{B}_{\theta}\) and let \(q\) be given above.
We take \(\mathfrak{H}u\) to be given by a smooth mollification at the scale of the distance to the boundary, 
given for example by \eqref{equationConformalExtension}.
By the Morrey--Sobolev embedding, we have 
\[
 \dist (\mathfrak{H}u (x), u (x/\abs{x}))
 \le C_1 \theta^{\frac{1}{p}} (1 - \abs{x})^{1 - \frac{n + 1}{p}}\ .
\]
We choose \(\rho \in (1/2, 1)\) such that 
\[
 C_1 \theta^{\frac{1}{p}} (1 - \rho)^{1 - \frac{n + 1}{p}} < \iota\ ,
\]
we define the function \(\eta : \Bset^{n + 1} \to \Rset\) for \(x \in \Bset^{n + 1}\) by 
\(
 \eta (x) = ((1 - \abs{x})/(1 - \rho))_+\)
and we set
\(
 U
 = \bigl(\eta \mathfrak{H}u + (1-\eta) U_q\bigr)\ .
\)
We observe that \(U \in W^{1, p} (\Bset^{n + 1}; M)\). Moreover
\[
\begin{split}
  \int_{\Bset^{n + 1}} \abs{D U}^p
  \le C_2 \Bigl(\int_{\Bset^{n + 1}} \abs{D \mathfrak{H} u}^p + \int_{\Bset^{n + 1}} \abs{D  U_q}^p\Bigr)\\
  \le C_2 \Bigl(C_3 \theta + \max_{j \in \{1, \dotsc, Q\}} \int_{\Bset^{n + 1}} \abs{DU_j}^p\Bigr)\ .\qedhere
\end{split}
\]
\end{proof}

\section*{Acknowledgement}

The authors thank Augusto Ponce for fruitful discussions about the problem.


\begin{bibdiv}

\begin{biblist}


\bib{Ahlfors1981}{book}{
   author={Ahlfors, Lars V.},
   title={M\"obius transformations in several dimensions},
   series={Ordway Professorship Lectures in Mathematics},
   publisher={University of Minnesota School of Mathematics},
   place={Minneapolis, Minn.},
   date={1981},
   pages={ii+150},
}

\bib{Bethuel1991}{article}{
   author={Bethuel, Fabrice},
   title={The approximation problem for Sobolev maps between two manifolds},
   journal={Acta Math.},
   volume={167},
   date={1991},
   number={3-4},
   pages={153--206},
   issn={0001-5962},
}

\bib{Bethuel2014}{article}{
   author={Bethuel, Fabrice},
   title={A new obstruction to the extension problem for Sobolev maps
   between manifolds},
   journal={J. Fixed Point Theory Appl.},
   volume={15},
   date={2014},
   number={1},
   pages={155--183},
   issn={1661-7738},
}

\bib{BethuelDemengel1995}{article}{
   author={Bethuel, F.},
   author={Demengel, F.},
   title={Extensions for Sobolev mappings between manifolds},
   journal={Calc. Var. Partial Differential Equations},
   volume={3},
   date={1995},
   number={4},
   pages={475--491},
   issn={0944-2669},
}

\bib{BourgainBrezisMironescu2005}{article}{
   author={Bourgain, Jean},
   author={Brezis, Ha{\"{\i}}m},
   author={Mironescu, Petru},
   title={Lifting, degree, and distributional Jacobian revisited},
   journal={Comm. Pure Appl. Math.},
   volume={58},
   date={2005},
   number={4},
   pages={529--551},
   issn={0010-3640},
}

\bib{BrezisMironescu2015}{article}{
   author={Brezis, Ha{\"{\i}}m},
   author={Mironescu, Petru},
   title={Density in \(W^{s,p}(\Omega;N)\)},
   journal={J. Funct. Anal.},
   volume={269},
   date={2015},
   number={7},
   pages={2045--2109},
   issn={0022-1236},
}

\bib{BrezisNirenberg1995}{article}{
   author={Brezis, H.},
   author={Nirenberg, L.},
   title={Degree theory and BMO. I. Compact manifolds without boundaries},
   journal={Selecta Math. (N.S.)},
   volume={1},
   date={1995},
   number={2},
   pages={197--263},
   issn={1022-1824},
}

\bib{CroweZweibelRosenbloom1986}{article}{
   author={Crowe, J. A.},
   author={Zweibel, J. A.},
   author={Rosenbloom, P. C.},
   title={Rearrangements of functions},
   journal={J. Funct. Anal.},
   volume={66},
   date={1986},
   number={3},
   pages={432--438},
   issn={0022-1236},
}

\bib{DiBenedetto2002}{book}{
   author={DiBenedetto, Emmanuele},
   title={Real analysis},
   series={Birkh\"auser Advanced Texts: Basler Lehrb\"ucher},
   publisher={Birkh\"auser},
   address={Boston, Mass.},
   date={2002},
   pages={xxiv+485},
   isbn={0-8176-4231-5},
}

\bib{GromovSchoen1992}{article}{
   author={Gromov, Mikhail},
   author={Schoen, Richard},
   title={Harmonic maps into singular spaces and \(p\)-adic superrigidity for
   lattices in groups of rank one},
   journal={Inst. Hautes \'Etudes Sci. Publ. Math.},
   number={76},
   date={1992},
   pages={165--246},
   issn={0073-8301},
}

\bib{HardtLin1987}{article}{
   author={Hardt, Robert},
   author={Lin, Fang-Hua},
   title={Mappings minimizing the \(L^p\) norm of the gradient},
   journal={Comm. Pure Appl. Math.},
   volume={40},
   date={1987},
   number={5},
   pages={555--588},
   issn={0010-3640},
}

\bib{Isobe2003}{article}{
   author={Isobe, Takeshi},
   title={Obstructions to the extension problem of Sobolev mappings},
   journal={Topol. Methods Nonlinear Anal.},
   volume={21},
   date={2003},
   number={2},
   pages={345--368},
   issn={1230-3429},
}
        
\bib{LiebLoss}{book}{
  author={Lieb, Elliott H.},
  author={Loss, Michael},
  title={Analysis},
  series={Graduate Studies in Mathematics},
  volume={14},
  edition={2},
  publisher={American Mathematical Society},
  place={Providence, RI},
  date={2001},
  pages={xxii+346},
  isbn={0-8218-2783-9},
}


\bib{mazya}{book}{
    AUTHOR = {Maz\cprime{}ya, Vladimir Gilelevich},
     TITLE = {Sobolev spaces with applications to elliptic partial
              differential equations},
    SERIES = {Grundlehren der Mathematischen Wissenschaften},
    VOLUME = {342},
 PUBLISHER = {Springer},
     address={Heidelberg},
      YEAR = {2011},
     PAGES = {xxviii+866},
      ISBN = {978-3-642-15563-5},
       DOI = {10.1007/978-3-642-15564-2},
       URL = {http://dx.doi.org/10.1007/978-3-642-15564-2},
}


\bib{Mostow1968}{article}{
   author={Mostow, G. D.},
   title={Quasi-conformal mappings in \(n\)-space and the rigidity of
   hyperbolic space forms},
   journal={Inst. Hautes \'Etudes Sci. Publ. Math.},
   number={34},
   date={1968},
   pages={53--104},
   issn={0073-8301},
}
\bib{Nash1956}{article}{
   author={Nash, John},
   title={The imbedding problem for Riemannian manifolds},
   journal={Ann. of Math. (2)},
   volume={63},
   date={1956},
   pages={20--63},
   issn={0003-486X},
}

\bib{Nicolesco1936}{article}{
    author = {Nicolesco, M.},
    title = {Expos\'es sur la th\'eorie des fonctions},
    partial={
    part={IV},
    subtitle={Les fonctions polyharmoniques},
    date = {1936},
    journal={Actual. sci. industr.},
    volume={331},
    },
}

\bib{PetracheRiviere2015}{article}{
   author={Petrache, Mircea},
   author={Rivi{\`e}re, Tristan},
   title={Global gauges and global extensions in optimal spaces},
   journal={Anal. PDE},
   volume={7},
   date={2014},
   number={8},
   pages={1851--1899},
   issn={2157-5045},
}



\bib{SchoenUhlenbeck}{article}{
   author={Schoen, Richard},
   author={Uhlenbeck, Karen},
   title={Boundary regularity and the Dirichlet problem for harmonic maps},
   journal={J. Differential Geom.},
   volume={18},
   date={1983},
   number={2},
   pages={253--268},
   issn={0022-040X},
}

\bib{SchulzeWildenhain1977}{book}{
   author={Schulze, Bert-Wolfgang},
   author={Wildenhain, G{\"u}nther},
   title={Methoden der Potentialtheorie f\"ur elliptische
   Differentialgleichungen beliebiger Ordnung},
   series={Lehrb\"ucher und Monographien aus dem Gebiete der Exakten
   Wissenschaften: Mathematische Reihe},
   number = {60},
   publisher={Birkh\"auser},
   address={Basel--Stuttgart},
   date={1977},
   pages={xv+408},
   isbn={3-7643-0944-X},
}
		



\bib{Stein1970}{book}{
   author={Stein, Elias M.},
   title={Singular integrals and differentiability properties of functions},
   series={Princeton Mathematical Series}, 
   volume={30},
   publisher={Princeton University Press},
   place={Princeton, N.J.},
   date={1970},
}


\end{biblist}

\end{bibdiv}
\end{document}